\documentclass[11pt]{article}
\usepackage{amsmath,amsfonts,latexsym, xspace,amsthm,graphicx, amssymb}
\usepackage{enumerate}
\usepackage[letterpaper,left=1in,right=1in,top=1in,bottom=1in]{geometry}

\usepackage{thmtools}
\usepackage{thm-restate}

\usepackage{mathtools}
\mathtoolsset{showonlyrefs}

\usepackage{color}

\usepackage{fancyhdr}
\newcommand{\E}[1]{\textbf{E} \left[#1\right]}
\def\EE{\mathcal{E}}

\def\LL{\mathcal{L}}

\def\mP{\mathcal{P}}

\def\mC{\mathcal{C}}
\def\mG{\mathcal{G}}

\def\a{\alpha}
\def\b{\beta}
\def\d{\delta}
\def\D{\Delta}
\def\e{\varepsilon}
\def\f{\phi}
\def\g{\gamma}
\def\G{\Gamma}

\def\l{\lambda}
\def\m{\mu}
\def\n{\nu}

\def\s{\sigma}

\def\t{\tau}
\def\om{\omega}
\def\OM{\Omega}

\newcommand\Prob[1]{{\mbox{Pr}\left\{#1\right\}}}

\newtheorem{lemma}{Lemma}
\newtheorem{theorem}{Theorem}

\newcommand{\brac}[1]{\left( #1\right)}
\newcommand{\bfrac}[2]{\brac{\frac{#1}{#2}}}

\title{Deletion of oldest edges in a preferential attachment graph}
\author{Tony Johansson\thanks{Research supported in part by NSF Grant  DMS1362785. Email: tjohanss@andrew.cmu.edu}\\ Department of Mathematical Sciences\\Carnegie Mellon University\\Pittsburgh PA 15213\\U.S.A.}

\begin{document}
\maketitle



\begin{abstract}
We consider a variation on the Barab\'asi-Albert random graph process with fixed parameters $m\in \mathbb{N}$ and $1/2 < p < 1$. With probability $p$ a vertex is added along with $m$ edges, randomly chosen proportional to vertex degrees. With probability $1 - p$, the oldest vertex still holding its original $m$ edges loses those edges. It is shown that the degree of any vertex either is zero or follows a geometric distribution. If $p$ is above a certain threshold, this leads to a power law for the degree sequence, while a smaller $p$ gives exponential tails. It is also shown that the graph contains a unique giant component with high probability if and only if $m\geq 2$.
\end{abstract}

\section{Introduction}\label{introsec}

In recent years, considerable attention has been paid toward real-world networks such as the World Wide Web (e.g. \cite{fff}) and social networks such as Facebook \cite{ukbm} and Twitter \cite{msgl}. Many but not all of these networks exhibit a so-called power law, and are sometimes referred to as scale free, meaning that the number of elements of degree $k$ is asymptotically $k^{-\eta}$ for some constant $\eta > 0$. In \cite{betal} it is shown that the social network of scientific collaborations is scale free. For a number of real-world scale free networks see \cite{betal}.

As a means of describing such networks with a random graph, Barab\'asi and Albert \cite{ba} introduced a class of models, commonly called preferential attachment graphs, and argued that its degree distribution has a tail that decreases polynomially, a claim that was subsequently proved by Bollob\'as, Riordan, Spencer and Tusn\'ady \cite{brst}. This is in contrast to many well-known random graph models such as the Erd\H{o}s-R\'enyi model where the degree distribution has an exponential tail. While the Barab\'asi-Albert model in its basic form falls short as a description of the World Wide Web \cite{ah}, the model has become popular for modelling scale free networks.

The base principle of preferential attachment graphs is the following: vertices are added sequentially to the graph, along with edges that attach themselves to previously existing vertices with probability proportional to their degree. This principle is susceptible to many variations, and can be combined with other random graph models. See for example Flaxman, Frieze and Vera \cite{ffv1}, \cite{ffv2}, who introduced a random graph model combining aspects of preferential attachment graphs and random geometric graphs.

Real-world networks will encounter both growth and deletion of vertices and edges. Bollob\'as and Riordan \cite{br} considered the effect of deleting vertices from the graph after it has been generated. Cooper, Frieze and Vera considered random vertex deletion \cite{cfv}, and Flaxman, Frieze and Vera considered adversarial vertex deletion \cite{ffv3}, where vertices are deleted while the graph is generated. Chung and Lu \cite{cl} considered a general growth-deletion model for random power law graphs.

In this paper, we consider a preferential attachment model in which the oldest edges are regularly removed while the graph is generated. There are two fixed parameters, an integer $m\geq 1$ and a real number $1/2 < p < 1$. As the graph is generated, with probability $p$ we add a vertex along with $m$ edges to random endpoints proportional to their degree. Choices are made with or without replacement. The vertices are ordered by time of insertion, and with probability $1 - p$ we remove all edges that were added along with a vertex, where the vertex is the oldest for which this has not already been done. This is a new variation of the preferential attachment model, and the focus on the paper is to find the degree sequence of this graph. The proof method also leads to a partial result on the existence of a giant component.

In Theorem \ref{thm2} we find the degree sequence of the graph, and show that it exhibits a phase transition at $p = p_0\approx 0.83$, independently of $m$. If $p > p_0$ then the degree sequence follows a power law, while $p < p_0$ gives exponential tails. A real-world example of this behaviour is given by family names; in \cite{mlnm} it is shown that the frequency of family names in Japan follow a power law, while \cite{kp} shows that family names in Korea decay exponentially.

We prove three theorems. The first deals with the degree distribution of any fixed vertex, show that it is the sum of $m$ independent variables that are either zero or geometrically distributed. We let $G_n$ denote the $n$th member of the graph sequence described above. The notation given here is imprecise at this point, but the theorems will be restated with precise notation below.

Let $\mathcal{D}$ be the event that at some point of the graph process, the graph contains no edges. The probability of $\mathcal{D}$ is addressed in Lemma \ref{sigmalemma}, and we will be conditioning on $\mathcal{D}$ not occurring. At this point we remark that if the process starts with a graph with $\om_H$ edges where $\om_H = \om_H(n)\to\infty$ as $n\to \infty$, then $\Prob{\mathcal{D}} = o(1)$. Note that the $\om$ in the following theorem is different from $\om_H$.
\begin{theorem}\label{thm1}
Suppose $\om = o(\log n)$ tends to infinity with $n$. Let $d(n, v)$ denote the degree of vertex $v$ in $G_n$. Conditioning on $\overline{\mathcal{D}}$, there exist functions $p(n, v), q(n, v)) \in [0, 1]$ and a constant $0 < c < 1/2$ such that $d(n, v)$ is distributed as the sum $d_1(n,v) +d_2(n, v) + \dots + d_m(n, v)$ of independent random variables with
$$
\Prob{d_i(n, v) = k} = \left\{\begin{array}{ll}
1 - q(n, v) + O(n^{-c}), & k = 0, \\
q(n, v)p(n, v)(1 - p(n, v))^{k - 1} + O(n^{-c}), & k > 0,
\end{array}\right.
$$
for $i = 1,\dots,m$, for all $v\geq n/\om$.
\end{theorem}
We do not address the degrees of vertices $v < n/\om$. In particular, we present no bounds for the maximum degree of $G_n$. We have instead focused on finding the degree sequence and connected components of $G_n$.

The second theorem translates the degree distribution into a degree sequence for $G_n$. It shows that the graph follows a power law if and only if $p$ is above a certain threshold.
\begin{theorem}\label{thm2}
Let $p_0 \approx 0.83113$ be the unique solution in $(1/2, 1)$ to $p / (4p - 2) = \ln(p / (1-p))$. Let $X_k(n)$ denote the number of vertices of degree $k$ in $G_n$. Conditioning on $\overline{\mathcal{D}}$, there exists a sequence $\{x_k : k\geq 0\}$ such that
\begin{enumerate}[(i)]
\item if $\a < 1$ then $x_k = \a^{k(1+o_k(1))}$ and if $\a > 1$ then there exist constants $a, b > 0$ such that $x_k = ak^{-\eta-1} + O_k(k^{-\eta-2}\log^b k)$, where $\eta = \eta(p) > 2$ is defined for $p > p_0$, and
\item for any fixed $k \geq 0$, $X_k(n) = x_kn(1+o_n(1))$ with high probability\footnote{We say that a sequence of events $\EE_n$ occur with high probability (whp) if $\Prob{\EE_n}\to 1$ as $n\to\infty$}.
\end{enumerate}
\end{theorem}
The third theorem shows that $G_n$ has a giant component if and only if $m \geq 2$. This is in contrast to the classical Barab\'asi-Albert model which is trivially connected. Let $B(n) = \l \ln n$ when $p < p_0$ and $B(n) = \l n^{1/\eta}\ln n$ when $p > p_0$ for some constants $\l > 0$, $\eta > 2$, explicitly defined later. Note that when $p > p_0$ and $m$ is large, Theorem \ref{thm3} states that the number of vertices which are not in the largest component is $O_m(c^m n)$ for some $0 < c < 1$, since the total number of vertices in $G_n$ will be shown to be $pn(1 + o_n(1))$ whp.
\begin{theorem}\label{thm3}
Condition on $\overline{\mathcal{D}}$.
\begin{enumerate}[(i)]
\item If $m = 1$, the largest component of $G_n$ has size $O(\D\log n)$ with high probability, where $\D$ is the maximum degree of $G_n$.

\item If $m \geq 2$, there exists a constant $\xi > 0$ such that with high probability the number of isolated vertices is $\xi pn$, the largest component contains at least $(1-\xi)(1-(13/14)^{m-1}) pn$ vertices, and all other components have size $O(\log n)$. If $p > p_0$ then $\xi = O_m(c^m)$ for some $0 < c < 1$.
\end{enumerate}
\end{theorem}

\addtocounter{theorem}{-3}

\subsection{Proof outline}\label{outlinesec}

The paper is laid out as follows. In Section \ref{modelsec} we define the graph process precisely and define constants and functions that are central to the main results. Section \ref{poissonsec} is devoted to Crump-Mode-Jagers processes, which will be the central tool in studying the graph process. Sections \ref{distributionsec}, \ref{degseqsec} and \ref{giantsec} are devoted to proving Theorem \ref{thm1}, \ref{thm2} and \ref{thm3} respectively.

We will now outline the proof of Theorem \ref{thm1}. Theorem \ref{thm2} is an elementary consequence of Theorem \ref{thm1}, and the proof of Theorem \ref{thm3} is heavily based on the method used to prove Theorem \ref{thm1}.

In Section \ref{mastersec} we will define a {\em master graph} $\G$ which contains $G_t$ for all $t$. We will mainly be proving results for $\G$ and show how they transfer to $G_n$, but for this informal outline we will avoid the somewhat technical definition of $\G$ and show the idea behind the proofs.

We begin by describing the Crump-Mode-Jagers process (or CMJ process). The name Crump-Mode-Jagers applies to a more general class than what is considered here, but we will mainly be talking about the special case described as follows. Fix a constant $\a > 0$ and consider a Poisson process $\mP_0$ with rate $\a$ on $[0, 1)$. Suppose $\mP_0$ has arrivals at time $\t_{01} < \t_{02} < \dots < \t_{0k}$. The $j$th arrival gives rise to a Poisson process $\mP_{0j}$ on $[\t_{0j}, \t_{0j} + 1)$, $j = 1,\dots, k$, independent of all other Poisson processes. In general, let $s = 0\dots$ be a string of integers starting with $0$ and suppose $\mP_s$ is a Poisson process on $[\t_s, \t_s + 1)$. Then the $j$th arrival in $\mP_s$, at time $\t_{sj}$, gives rise to a Poisson process $\mP_{sj}$ on $[\t_{sj}, \t_{sj} + 1)$. Here $sj$ should be interpreted as appending $j$ to the end of the string $s$. Let $d(\t)$ be the number of processes alive at time $\t$, i.e. the number of processes $\mP_s$ with $\t_s\in (\t-1, \t]$. Lemma \ref{CMJlemma} will show that for fixed $\t$, $d(\t)$ is either zero or geometrically distributed.

We will now explain how the degree of a vertex in the graph process relates to a CMJ process. Firstly, note that choosing a random vertex with probability proportional to degrees is equivalent to choosing an edge $e$ uniformly at random, and choosing one of the two endpoints of $e$ uniformly at random. We will refer to this as choosing a {\em half-edge} $(e, \ell)$ where $\ell\in\{1,2\}$. If $e = \{v, w\}$ is added along with vertex $v$, we say that choosing $(e, 1)$ corresponds to choosing $w$ {\em via} $e$, and choosing $(e, 2)$ corresponds to choosing $v$ via $e$. This is described in detail in Section \ref{modelsec}.

For the purpose of demonstration consider the case $m = 1$, i.e. the case in which exactly one edge is added along with any vertex added to the graph. It will follow from Lemma \ref{sigmalemma} that if a vertex $v_0$ is added along with an edge $e_0$ at time $t_0$ then with high probability $e_0$ is removed at time $\g t_0 + o(t)$, where $\g = p / (1-p)$. Note that the degree of $v_0$ may still be non-zero after the removal of $e_0$. If the degree of $v_0$ is to increase from its initial value $1$, then there must exist a time $T_{01}$ with $t_0 < T_{01} < \g t_0 + o(t_0)$ at which a vertex $v_{01}$ is added along with edge $e_{01}$, where $e_{01}$ is randomly assigned to $(e_0, 2)$. The time $T_{01}$ is random and we will see (equation \eqref{nearexp}) that $\log_\g (T_{01} / t_0) \in (0, 1+o(1))$ is approximately exponentially distributed with rate $\a = \a(p)$. Furthermore, if $T_{01} < T_{02} < \dots < T_{0k}$ denote the times at which a vertex is added that chooses $v_0$ via $e_0$, then the sequence $(\log_\g(T_{0i} / t_0))$ can be approximated by a Poisson process with rate $\a$ on the interval $(0, 1)$. Let $e_{01}$ denote the edge that is added at time $T_{01}$ and chooses $(e_0, 2)$ (if such an edge exists). Then the degree of $v_0$ may increase by some edge $e_{011}$ added at time $T_{011}$ with $T_{01} < T_{011} < \g T_{01} + o(T_{01})$ choosing $(e_{01}, 1)$, i.e. choosing $v_0$ via $e_{01}$. As above, the sequence of times $T_{011}, T_{012}, \dots, T_{01\ell}$ at which a vertex is added that chooses $v_0$ via $e_{01}$ are such that $(\log_\g(T_{01i} / T_{01}))$ approximately follows a Poisson process on $(\log_\g T_{01}, 1 + \log_\g T_{01})$. Repeating the argument, any edge incident to $v_0$ gives rise to a Poisson process, and as long as the degree of $v_0$ is not too large the processes are ``almost independent''. Under the time transformation $\t(t) = \log_\g (t / t_0)$, the times at which the degree of $v_0$ increases or decreases can be approximated by the times at which $d(\t)$ increases or decreases in a CMJ process with rate $\a$. This approximation is made precise in the proof of Theorem \ref{degthm}, and shows that the degree of a vertex is either zero or approximately geometrically distributed.

Now suppose $m > 1$. Then each of the $m$ edges added along with $v$ gives rise to a CMJ process by the argument above, and the processes are ``almost independent''. The degree of $v$ will be approximated by a sum of $m$ independent random variables that are each either zero or geometrically distributed.

\section{The model} \label{modelsec}

Fix $m \in \mathbb{N}$ and $1/2 < p < 1$. Let $\mG_m$ be the class of undirected graphs on $[\nu_G] = \{1,\dots, \nu_G\}$ for some integer $\nu_G$ such that if edges are oriented from larger integers to smaller, there exists some integer $1_G$ with $m < 1_G \leq \nu_G$ such that a vertex $v$ has out-degree $m$ if $v \geq 1_G$ and out-degree zero if $v < 1_G$. All graphs we deal with will be in $\mG_m$. In some places it will be convenient to think of graphs as being directed, in which case we always refer to the orientation from larger to smaller integers. We will allow parallel edges but no self-loops.

Our graph $G$ will be defined by $G = G_n$ for some graph sequence $(G_t)$ and some $n$ that grows to infinity. Each $G_t$ will be in $\mG_m$, and we write $1_t = 1_{G_t}, \nu_t = \nu_{G_t}$. Given $G_t$, we randomly generate $G_{t + 1}$ as follows. With probability $1 - p$, remove all $m$ edges oriented out of $1_t$, so that $1_{t + 1} = 1_t + 1$. Note that edges oriented into $1_t$ remain in $G_{t + 1}$. With probability $p$, add vertex $\nu_{t+1} = \nu_t + 1$ along with $m$ edges to distinct vertices, where vertices are picked with probability proportional to their degree with replacement. In other words, if $d(t, v)$ denotes the degree of vertex $v$ in $G_t$, then $\nu_{t+1}$ is added along with edges $(\nu_{t+1}, v_i)$ where $v_1,\dots, v_m$ are independent with
$$
\Prob{v_i = v} = \frac{d(t-1,v)}{e(G_{t-1})}
$$
where $e(G_{t-1})$ denotes the number of edges in $G_{t-1}$. Rather than using the terminology of $\nu_{t+1}$ choosing $v_1,\dots, v_m$, we will say that the $m$ edges $e_1,\dots, e_m$ added at time $t+1$ choose $v_1,\dots,v_m$ respectively. Let $d^+(t, v), d^-(t, v)$ denote the out- and in-degree of $v$ in $G_t$ in the natural orientation. Write $d(t,v) = 0$ if $v\notin G_t$. The issue of the empty graph appearing in the process is addressed shortly.

We will assume that the graph process starts with some graph $H\in \mG_m$ on $\nu = o(n^{1/2})$ vertices, and we label this graph $G_{t_0}$ where $t_0 = 1_H + \nu_H$ in order to  maintain the identity $1_t + \nu_t = t$ for every $G_t$, $t_0 \leq t \leq n$. Let $\s \in \{0, 1\}^{n-t_0}$ be such that $\s(u)$ is the indicator for if a vertex and $m$ edges are added at time $u + t_0$, or if $m$ edges are removed at time $u+t_0$. Then $\nu_t = \nu_H + \sum_{u = t_0 + 1}^t \s(u)$ for all $t > t_0$, and $1_t = t - \nu_t$. The entries $\s(u)$ are independent and $\s(u) = 1$ with probability $p$. Say that $\s$ is {\em feasible} if it is such that $\nu_t > 1_t$ for all $t > t_0$, noting that $\{\s$ is feasible$\} = \overline{\mathcal{D}}$ with $\mathcal{D}$ as in Section \ref{introsec}. For a function $\om = \om(n)$ such that $\om \to\infty$ as $n\to\infty$, We say that $\s$ is {\em $\om$-concentrated} if $|\nu_t - pt| \leq t^{1/2}\ln t$ for all $t\geq n/\om$. Note that if $\s$ is $\om$-concentrated then $|1_t - (1-p)t| \leq t^{1/2}\ln t$ and $|e(G_t)-m(2p-1)t| \leq mt^{1/2}\ln t$ for all $t\geq n/\om$. Furthermore, if an edge $e$ is added at time $t\geq n/\om$ then it is removed at time $pt / (1-p) + O(t^{1/2}\ln t)$.
\begin{lemma}\label{sigmalemma}
Let $\om = \om(n) \to \infty$ with $n$. If the graph process is initiated at $H\in\mG_m$ on $\nu_H \leq \om^{-1}n^{1/2}$ vertices and $\nu_H - 1_H = N$, then $\s$ is feasible with probability $1 - O(c^N)$, i.e. $\Prob{\mathcal{D}} = O(c^N)$, for some constant $c\in (0, 1)$. Furthermore, $\s$ is $\om$-concentrated with probability $1-O(n^{-C})$ for any $C > 0$.
\end{lemma}

\begin{proof}
Recall that $\s(t) = 1$ with probability $p$ and $\s(t) = 0$ otherwise. The difference $\nu_t - 1_t$ is a random walk, and the fact that $\nu_t\geq 1_t$ for all $t\geq t_0$ with probability $1-O(c^N)$ for some $c\in (0, 1)$ is well known (see e.g. \cite[Section 5.3]{grimmett}).

Suppose $t\geq n/\om$. Then $t-t_0 \geq n / 2\om$ and by Hoeffding's inequality \cite{hoef}, since $pt_0 - \nu_H = o((n/\om)^{1/2}) = o(t^{1/2}\ln t)$,
\begin{align*}
\Prob{\nu_t - pt > t^{1/2}\ln t} = \Prob{\left(\sum_{u=t_0+1}^t \s(u)\right) - p(t-t_0) > pt_0 - \nu_H + t^{1/2}\ln t} = e^{-\Omega(\ln^2 t)}
\end{align*}
Summing over $t = n/\om, \dots, n$ shows that $\nu_t \leq pt + t^{1/2}\ln t$ for all $t\geq n/\om$ whp, and similarly $\nu_t \geq pt - t^{1/2}\ln t$ for all $t\geq n/\om$ whp.
\end{proof}

\subsection{The master graph}\label{mastersec}

The above description of $G_t$ is limited in that it forces one to generate the graph on-line, i.e. vertices need to make their random choices in a fixed order. Conditioning on $\s$ we can define an off-line graph $\G$ which contains $G_t$ for all $t$. This graph enables us to generate small portions of the graph without revealing a large part of the probability space.

Fixing a feasible $\s$ we define a {\em master graph} $\G = \G_n^\s(H)$ which has $G_t$ as a subgraph (in the sense that $G_t$ can be obtained from $\G$ by removing edges and possibly vertices) for all $t_0 \leq t \leq n$. There are two key observations that allow the construction. Firstly, if $\s$ is fixed, then $\nu_t = \nu_H + \sum_{u = 1}^{t-t_0}\s(u)$ is known for all $t_0 \leq t \leq n$. This means that $1_t = t - \nu_t$ is known, and we know that the $m(\nu_t - 1_t + 1)$ edges in $G_t$ are those added along with $1_t, 1_t + 1,\dots,\nu_t$, for all $t_0\leq t\leq n$. Secondly, suppose a vertex $v$ is added along with edges $e_1,\dots, e_m$ at time $t > t_0$. Rather than using the terminology of $v$ choosing $m$ vertices $v_1,\dots,v_m\in G_{t-1}$ with probability proportional to their degrees, we will adopt the terminology of the edges $e_i$ independently choosing edges $f_i\in G_{t-1}$ uniformly at random, then choosing one of the two endpoints of $f_i$ uniformly at random. To make this formal, let $E_{e}^\s$ be the edges that are in the graph when the edge $e$ is added, noting that if $e$ is added at time $t$ then $E_e^\s = \{m(1_{t-1} - 1) + 1, \dots, m\nu_{t-1}\}$ with $1_{t-1}, \nu_{t-1}$ determined by $\s$. Then each $e_i$ independently chooses an $f_i\in E_{e_i}^\s$ uniformly at random, along with $j_i\in [2]$ chosen uniformly at random. If $f_i = \{u, u'\}$ with $u' < u$, then $e_i$ choosing $(f_i, 1)$ means $e_i$ chooses $u'$, and $(f_i, 2)$ means choosing $u$. We say that $f_i$ chooses $u$ (or $u'$) {\em via} $f_i$. We call a pair $(e, j)$ with $j\in [2]$ a {\em half-edge}.

Suppose the graph process is initiated with some graph $G_{t_0} = H \in \mG_m$ on $[\nu_H]$ with $1_H + \nu_H = t_0$. We will introduce an integer labelling $L(e)$ for the edges $e$ in $\G$. The $L$ will be dropped from calculations and we write $e_1 \geq e_2$ to mean $L(e_1) \geq L(e_2)$ and $f(e) = f(L(e))$ whenever $f$ is a function on the integers. The labelling is defined by labelling the $m$ edges along with $v > \nu_H$ by $m(v-1) + 1, m(v-1) + 2,\dots, mv$. The edges in the initial graph $H$ can be oriented in such a way that vertices $1, \dots, 1_H - 1$ have out-degree zero, and $1_H, \dots, \nu_H$ have out-degree $m$. We can then label the edges in $H$ by $m(1_H - 1) + 1, \dots, m\nu_H$ in such a way that $1_H \leq v\leq \nu_H$ is incident with edges $m(v-1) + 1,\dots, mv$. Note that under this labelling, every edge $e$ is incident with vertex $\lceil e / m\rceil$ while its other endpoint $v(e)$ will satisfy $v(e) < \lceil e / m\rceil$.

{\bf Definition of $\G$:} Fix a feasible $\s$ and a graph $H\in \mG_m$. We define $\G = \G_n^\s(H)$ as follows. The vertex set is $[\nu_n]$ where $\nu_n = \nu_H + \sum_{i = 1}^{n - t_0} \s(i)$. The graph $\G$ contains $H$ as an induced subgraph on $[\n_H]$. Every edge $e > m\nu_H$ is associated with a set $\OM(e) = E_e^\s\times [2]$, and makes a random choice $\f(e) = (f(e), j(e)) \in \OM(e)$ uniformly at random, independent of all other edges. One endpoint of $e$ is $\lceil e / m\rceil$ (the fixed endpoint) and one is $v(e)$ (the random endpoint). If $j(e) = 2$ then $v(e) = \lceil f(e) / m \rceil$. If $j(e) = 1$ then $v(e) = v(f(e))$.

Note the recursion in defining the random endpoint $v(e)$ of an edge $e$. If $j(e) = 1$ and $j(f(e)) = 1$ then $v(e) = v(f(e)) = v(f(f(e)))$, and so on until either $j(f^{(k)}(e)) = 2$ for some $k$, or $f^{(k)}(e)\leq m\nu_H$ for some $k$, in which case $v(e) = v(f^{(k)}(e))$ is determined by $H$. Here $f^{(k)}$ denotes $k$-fold composition of $f$.

We will generate $\G$ carefully by keeping a close eye on the sets $\OM(e)$. Let $\G_0$ be the graph in which no randomness has been revealed; in $\G_0$ only the edges in $H$ are known, all other edges are free, and all sets $\OM(e) = \OM_0(e) = E_e^\s\times [2]$. For sets $A \subseteq \{m\nu_H + 1, m\nu_H + 2, \dots, m\nu_n\}$ of free edges and $R \subseteq \{m(1_H - 1) + 1, m(1_H - 1) + 2, \dots, m\nu_n\} \times [2]$ of half-edges, define a class $\mG(A, R)$ of {\em partially generated}  graphs as follows. We say that $\widetilde{\G} \in \mG(A, R)$ if (i) for $e > m\n_H$, $\f(e)$ is known if and only if $e\in A$, and (ii) for all $e\notin A$ we have $\OM(e) \supseteq \OM_0(e)\setminus R$. In other words, if $(f, j)\in R$ then for each $e$ with $(f,j)\in \OM_0(e)$, we may have determined that $\f(e)\neq (f,j)$.

Given a partially generated $\widetilde{\G} \in \mG(A, R)$, we define two operations that reveal more information about $\G$. We say that we {\em assign} $e \notin A$ when we choose $\f(e)$ uniformly at random from $\widetilde{\OM}(e) = \OM_0(e) \setminus R$. For any $(f, j)$ we can {\em reveal} $(f, j)$ to find the $\f^{-1}(\{f, j\}) \setminus A$ of edges $e\notin A$, free in $\widetilde{\G}$, that choose $(f, j)$. We reveal $(f, j)$ as follows. For every edge $e\notin A$ with $(f, j) \in \widetilde{\OM}(e)$, set $\f(e) = (f, j)$ with probability $1 / |\widetilde{\OM}(e)|$, and otherwise remove $(f, j)$ from $\widetilde{\OM}(e)$. 

Starting with $\G_0 \in \mG(\emptyset, \emptyset)$ (this class contains only one graph), we can generate $\G$ by a sequence of {\em assigns} and {\em reveals}. Given $\G_i \in \mG(A_i, R_i)$, we can assign $e\notin A_i$ to form $\G_{i+1} \in \mG(A_i \cup \{e\}, R_i)$, and we set $A_{i + 1} = A_i \cup \{e\}$ and $R_{i + 1} = R_i$. If $(f, j)\notin R_i$ is revealed and $e_1,\dots, e_k$ are the edges that choose $(f, j)$, we get $\G_{i+1}\in \mG(A_i \cup \{e_1, e_2,\dots, e_k\}, R_i \cup \{(f, j)\})$, and we set $A_{i + 1} = A_i \cup \{e_1,\dots, e_k\}$ and $R_{i+1} = R_i \cup \{(f, j)\}$. We get a sequence $\G_0, \G_1,\dots$ where $\G_i \in \mG(A_i, R_i)$ and $A_i \subseteq A_{i+1}$ and $R_i \subseteq R_{i+1}$ for all $i$.

Note that in a partially generated graph, if $\f(e) = (f(e), 1)$ where $f(e)$ is free, then we know that $v(e) = v(f(e))$, but $v(f(e))$ is not yet determined. We say that $e$ is committed to $f(e)$. This can be pictured by gluing the free end of $e$ to the free end of $f(e)$. At a later stage when $f(e)$ is attached to the its random endpoint $v(f(e))$, the edge $e$ will follow and be attached to the same vertex.


We will condition on $\s$ being $\om$-concentrated for some $\om$ in the proofs to follow. In $\G$, this translates to each $e$ with $e \geq mn/\om$ having $E_e^\s = \{e/\g+ O(n^{1/2}\ln n), \dots, m(\lceil e/m\rceil - 1)\}$, where $\g = p/(1-p)$. Note in particular that $|E_e^\s| = e(1 - 1/\g) + O(n^{1/2}\ln n)$. Note also that for any edge $e\geq mn/\om$, the largest $f$ for which $e\in E_f^\s$ is $f = \g e + O(n^{1/2}\ln n)$.

\subsection{Constants and functions}\label{confun}

In this section we collect constants and functions that will be used throughout the remaining sections. Fixing $p$ and $m$, we define
$$
\m = m(2p - 1), \quad \g = \frac{p}{1 - p}, \quad \a = \frac{pm}{2\m}\ln \g = \frac{p}{4p-2}\ln\g.
$$
The constant $\a$ will play a central role in what follows. We note that it is independent of $m$, and viewed as a function of $p\in(1/2, 1)$ it is continuously increasing and takes values in $(1/2, \infty)$. Let $p_0 \approx 0.83113$ be the unique $p$ for which $\a = 1$. When $\a \neq 1$ define $\zeta$ as the unique solution in $\mathbb{R} \setminus\{1\}$ to
$$
\zeta e^{\a(1-\zeta)} = 1.
$$
Also let $\eta = -\ln \g / \ln \zeta$ if $\a > 1$. If $\a < 1$ then $\eta$ is undefined.

Define a sequence $a_k$ by $a_0 = 1$ and
$$
a_k = \left(-\frac{e^\a}{\a}\right)\left(\frac{a_0}{(k-1)!} + \frac{a_1}{(k-2)!} + \dots + a_{k-1}\right) = \left(-\frac{e^\a}{\a}\right)\sum_{j=0}^{k-1} \frac{a_{j}}{(k-j-1)!}, \ k \geq 1.
$$
For $k \geq 0$ define functions $Q_k : [k, k + 1) \to [0, 1]$ by
$$
Q_k(\t) = \sum_{j = 0}^k \frac{a_j}{(k-j)!} (\t - k)^{k-j},
$$
and for $\t \geq 0$ let $Q(\t) = Q_{\lfloor \t\rfloor}(\t)$. We note that $Q(\t)$ is discontinuous at integer points $k$ with
$$
Q(k) = a_k \quad \mbox{and} \quad Q(k^-) = -\a e^{-\a} a_k
$$
where $Q(k^-)$ denotes the limit of $Q(\t)$ as $\t \to k$ from below. Define
$$
q(\t) = 1, \ 0 \leq\t<1, \quad q(\t) = 1 + \frac{Q(\t-1)}{\a Q(\t)}, \ \t \geq 1.
$$
Finally, define
$$
p(\t) = \exp\left\{-\a\int_0^\t q(x)dx\right\}.
$$
For $\t < 0$ we define $Q(\t) = q(\t) = p(\t) = 0$.

In Section \ref{giantsec} we will need explicit formulae for $q(\t)$ for $0\leq \t \leq 3$. We have $a_0 = 1$, $a_1 = -e^\a / \a$ and $a_0 = e^{2\a}/\a^2 - e^\a / \a$, so if $0 \leq \t < 1$ then $Q(\t) = 1$, $Q(\t + 1) = \t - e^\a / \a$ and $Q(\t + 2) = \frac12 \t^2 - \a^{-1}e^\a \t + e^{2\a}/\a^2 - e^\a / \a$, and
$$
q(\t) = 1, \quad q(\t + 1) = 1 - \frac{1}{e^\a - \a\t}, \quad q(\t + 2) = 1 - \frac{e^\a - \a\t}{e^{2\a} - (\t + 1)\a e^\a + \frac12 \a^2\t^2}, \ 0\leq \t < 1.
$$

The following lemma collects properties of the constants and functions presented here. Its proof is postponed to Appendix \ref{constlemproofsec}.

\begin{restatable}{lemma}{constlem}\label{constlem}
\begin{enumerate}[(i)]
\item \label{zetalemma} If $\a > 1$ then $\zeta < \a^{-1}$ and if $\a < 1$ then $\zeta > 1 - \a^{-1} + \a^{-2} > \a^{-1}$.
\item \label{etalemma} If $\a > 1$ then $\eta > 2$.
\item \label{decreasing} The functions $p(\t), q(\t)$ are decreasing and take values in $[0, 1]$.
\item \label{Qeat} For any non-integer $\t > 0$,
$$
Q'(\t) = Q(\t-1) \quad \mbox{and} \quad q(\t) = \frac1\a \frac{(Q(\t)e^{\a\t})'}{Q(\t) e^{\a\t}}.
$$
\item \label{smallalpha} If $\a < 1$ then there exist constants $\l_1, \l_2 > 0$ where $\l_1 < \a$ such that for all $\t \geq 0$,
$$
p(\t) = 1 - \alpha + \frac{\l_1}{\zeta^\t} + O(\zeta^{-2\t})
\quad\mbox{and}\quad
q(\t) = \frac{\l_2}{\zeta^\t} + O(\zeta^{-2\t}).
$$
\item \label{largealpha} If $\a > 1$ then there exist constants $\l_3, \l_4 > 0$ and a constant $C > 0$ such that for all $\t\geq 0$,
$$
\l_3 \zeta^\t \leq p(\t) \leq \l_3\zeta^\t + C \zeta^{2\t}
\quad\mbox{and}\quad
q(\t) = 1 - \zeta + \l_4 \zeta^\t + O(\zeta^{2\t}).
$$
\end{enumerate}
\end{restatable}
The proof of Lemma \ref{constlem} is postponed to Appendix \ref{constlemproofsec}.

\section{A Poisson branching process}\label{poissonsec}

We now define a process $\mC$, called a Crump-Mode-Jagers (or CMJ) process. The name Crump-Mode-Jagers applies to a more general class than what is considered here, but we will mainly be talking about the special case described as follows. Fix a constant $\a > 0$ and consider a Poisson process $\mP_0$ on $[0, 1)$. Suppose $\mP_0$ has arrivals at time $\t_{01} < \t_{02} < \dots < \t_{0k}$. The $j$th arrival gives rise to a Poisson process $\mP_{0j}$ on $[\t_{0j}, \t_{0j} + 1)$, $j = 1,\dots, k$, independent of all other Poisson processes. In general, let $s = 0***$ be a string of integers and suppose $\mP_s$ is a Poisson process on $[\t_s, \t_s + 1)$. Then the $j$th arrival in $\mP_s$, at time $\t_{sj}$, gives rise to a Poisson process $\mP_{sj}$ on $[\t_{sj}, \t_{sj} + 1)$. Here $sj$ should be interpreted as appending $j$ to the end of the string $s$. Let $d(\t)$ be the number of processes alive at time $\t$, i.e. the number of processes $\mP_s$ with $\t_s\in (\t-1, \t]$, and define $b(\t)$ to be the number of processes born before $\t$, i.e. the number of $s$ for which $\t_s \leq \t$.

For a random variable $X$ and $p,q\in [0,1]$, say that $X\sim G(p,q)$ if
$$
\Prob{X = k} = \left\{\begin{array}{ll}
1-q, & k = 0, \\
qp(1-p)^{k-1}, & k \geq 1.
\end{array}\right.
$$

\begin{restatable}{lemma}{CMJlemma}\label{CMJlemma}
For all $\t \geq 0$, $d(\t) \sim G(p(\t), q(\t))$.
\end{restatable}
The proofs of Lemmas \ref{CMJlemma} and \ref{CMJbirthlem} are postponed to Appendix \ref{CMJproofsec}.
\begin{restatable}{lemma}{CMJbirthlem}\label{CMJbirthlem}
There exists a constant $\l > 0$ such that for $0 \leq \t \leq \log_\g n$, as $n\to \infty$
\begin{enumerate}[(i)]
\item if $\a < 1$,
$$
\Prob{b(\t) > \l \ln n} = o\bfrac1n.
$$
\item if $\a > 1$,
$$
\Prob{b(\t) > \l n^{1/\eta} \ln n} = o\bfrac1n
$$
where $\eta = - \ln\g / \ln\zeta > 2$.
\item If $\a\neq 1$ then $d(\t) \geq \lfloor b(\t) / (\l\log_\g^2 n)\rfloor$ for all $0\leq\t\leq\log_\g n$ with probability $1 - o(n^{-1})$.
\end{enumerate}
\end{restatable}

Let $\l > 0$ be as provided by Lemma \ref{CMJbirthlem} and define
$$
B(n) = \left\{\begin{array}{ll}
\l\ln n, & \a < 1, \\
\l n^{1/\eta}\ln n, & \a > 1.
\end{array}\right.
$$
Given a time $\t > 0$ we can calculate $b(\t), d(\t)$ by the following algorithm, based on the breadth-first-search algorithm. Here $S, S'$ are sets of integer strings. The numbers $i, j$ count the number of times $\EE, \LL$ have been called, respectively.
\begin{enumerate}
\item[0.] Let $S = \{0\}$, $S' = \{0\}$ and $\t_0 = 0$.
\item If $S'$ is empty, stop and output $S$ and $T = \{\t_s : s\in S\}$. Otherwise choose the smallest $s\in S'$ (ordered lexicographically) and remove it from $S'$. Let $L_s = 1$ be the lifetime of process $\mP_s$.
\item Let $X_{s1}, X_{s2}, \dots, X_{s(k+1)}$ be independent Exp$(\a)$ variables where $k \geq 0$ is the smallest integer for which $X_{s1} + \dots + X_{s(k+1)} > L_s$. If $k \geq 1$, set
\begin{align*}
\t_{s1} & = \t_s + X_{s1}, \\
\t_{s2} & = \t_s + X_{s1} + X_{s2}, \\
& \vdots \\
\t_{sk} & = \t_s + X_{s1} + X_{s2} + \dots + X_{sk}.
\end{align*}
Add $s1, s2,\dots, sk$ to $S$ and $S'$.
\end{enumerate}

\section{The degree distribution}\label{distributionsec}

This section is devoted to proving the following theorem. Suppose the graph process starts at $H = G_{t_0}$ where  $t_0 = o(n^{1/2})$. As in Section \ref{modelsec}, let $\mathcal{D}$ denote the event that $1_t > \nu_t$ for some $t\geq t_0$.

Let $G^m(p, q)$ denote the distribution of $X = X_1 + X_2 + \dots + X_m$ where $X_1,\dots,X_m$ are independent $G(p, q)$ distributed variables. Let $g^m(k; p, q)$ denote the probability mass function of $X$. Note that we define $d(n, v) = 0$ if $v$ is not in $G_n$. The functions $p(\t), q(\t)$ are defined in Section \ref{confun}.
\begin{theorem}\label{degthm}
Let $\om = o(\log n)$ be such that $t_0 \leq n^{1/2} / \om$ and $\om\to\infty$ as $n\to \infty$. Let $v \geq n/\om$, $\d = n^{-1/2}\ln n$, and $\t = \log_\g (pn/v)$. There exists a function $\widetilde{q}(\t) \in [0,1]$ with $\widetilde{q}(\t) = q(\t)$ for all $\t \notin (-\d, \d)\cup (1-\d, 1+\d)$, such that the degree $d(n, v)$ of $v$ satisfies
$$
\Prob{d(n, v) = k \mid \overline{\mathcal{D}}} = g^m(k; p(\t), \widetilde{q}(\t)) + O(B(n)n^{-1/2}\ln^2 n), \ k \geq 0.
$$
\end{theorem}

In Section \ref{outlinesec}, the idea behind this proof is outlined in the notation of the process $G_t$. The full proof presented here is based on the master graph $\G$ and will be rather technical, but the idea is the same. Condition on a feasible and $\om$-concentrated $\s\in \{0,1\}^{n-t_0}$, see Lemma \ref{sigmalemma}. We will be considering the master graph $\G = \G_n^\s(H)$ defined in Section \ref{mastersec}. Let $E_n$ be the set of edges in $\G$ with at least one endpoint in $\{1_n,\dots, \nu_n\}$, so that $G_n$ is obtained from $\G$ by removing all edges not in $E_n$.

Consider the graph $\G_0 \in \mG(\emptyset, \emptyset)$ in which all edges $e > m\n_H$ are free. Fix a vertex $v > n/\om$, and let $e_\ell = m(v-1) + \ell, \ell = 1,\dots,m$ denote the $m$ edges adjacent to $v$. Suppose an edge $e > mv$ is adjacent to $v$ in $\G$. Then $e$ must choose $\f(e) = (f, j)$ for some edge $f$ which is also adjacent to $v$. Here $j$ must be $2$ if $f\in \{e_1, \dots, e_m\}$ and $1$ otherwise. In words, for an edge to be adjacent to $v$ in $\G$ but not in $\G_0$, it must choose the appropriate endpoint of some other edge that is adjacent to $v$ in $\G$.

We will now make this idea more precise. Consider a partially generated graph $\widetilde{\G} \in \mG(A, R)$ for some sets $A, R$. For $(e_0, j_0) \notin R$, we define an operation called {\em exposing $(e_0, j_0)$}, as a sequence of {\em reveals} (as defined in Section \ref{mastersec}). Let $Q_0 = \{(e_0, j_0)\}$. For $i \geq 1$ define $Q_i = \{(e, 1) : e\notin A, \f(e) \in Q_{i-1}\}$. Consider the following algorithm for finding the edges in $\cup_{i\geq 0} Q_i$. The parts labelled Setup are not essential to the running of the algorithm, but are included to emphasize the similarity to the algorithm in Section \ref{poissonsec}, to which it will later be compared.

\begin{figure}
\begin{center}
\includegraphics[width = 0.8\textwidth]{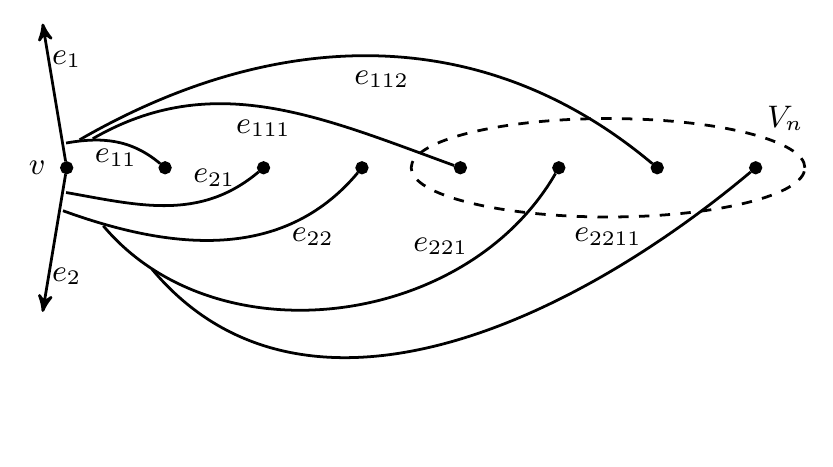}
\end{center}
\caption{One outcome of the expose algorithm for $m = 2$. Here $v$ has degree $9$ in $\G$ and degree $4$ in $G_n$. All edges in the figure are adjacent to $v$, but are drawn to indicate which half-edge was chosen, e.g. $\f(e_{221}) = (e_{22}, 1)$. Free edges are drawn as arrows.}
\label{fig:expose}
\end{figure}

The algorithm takes as input sets $A, R$, a partially generated $\widetilde{\G}\in \mG(A, R)$ and a half-edge $(e_0, j_0) \notin R$.
\begin{enumerate}
\item[0.] Let $S = \{0\}, S' = \{0\}$. Let $Q = \{(e_0, j_0)\}$.

\item If $S'$ is empty, stop and output $S$ and $Q$. Otherwise, let $s$ be the smallest member of $S'$ (in the lexicographical order) and remove $s$ from $S'$.

{\bf Setup:} Let $L_s' = \log_\g(f / e_s)$ where $f$ is the largest edge with $e_s \in E_f^\s$. 

\item Reveal $(e_s, j_s)$ to find $\f^{-1}(\{e_s, j_s\})\setminus A$. Label the edges in $\f^{-1}(\{e_s, j_s\})\setminus A$ by $e_{s1} < e_{s2} < \dots < e_{sk}$ (where $si$ denotes string concatenation). Add $(e_{s1}, 1), \dots, (e_{sk}, 1)$ to $Q$, and add $s1, s2,\dots,sk$ to $S$ and $S'$.

The partially generated graph is now in $\mG(A\cup \{e_{s1},\dots,e_{sk}\}, R\cup \{(e_s, j_s)\})$. Set $A \leftarrow A\cup \{e_{s1},\dots,e_{sk}\}$ and $R\leftarrow R\cup\{(e_s, j_s)\}$. Go to step 1.

{\bf Setup:} Let $X_{s1}' = \log_\g(e_{s1} / e_s)$ and $X_{s\ell}' = \log_\g(e_{s\ell} / e_{s(\ell-1)})$ for $\ell = 1,2,\dots, k$. Set $X_{s(k+1)}' = \infty$ and $e_{s(k+1)} = \infty$.
\end{enumerate}
With input $(e_0, j_0)$ and $\widetilde{\G} \in \mG(A, R)$, let $E((e_0, j_0), \widetilde{\G})$ be the set of edges $e\in E_n$ (i.e. edges in $G_n$) such that $e = e_s$ for some $s\in S$. 
\begin{lemma}\label{exposelemma}
Let $\om = o(\log n)$ tend to infinity with $n$. Suppose either $\a < 1$ and $0 < \e<1/2$, or $\a > 1$ and $0 < \e < 1/2-1/\eta$. Let $\widetilde{\G} \in \mG(A,R)$ where $|A|, |R| = O(n^{1/2 + \e}\log^k n)$ for some $k\geq 1$. Let $(e_0, j_0) \notin R$ satisfy $e_0 \geq mn/\om$, and let $\t = \log_\g(pmn / e_0)$. There exists a $\d = O(n^{-1/2}\ln n)$ and a function $\widetilde{q}(\t)\in [0, 1]$ such that
$$
\Prob{|E((e_0, j_0), \widetilde{\G})| = k} = \left\{\begin{array}{ll}
1 - \widetilde{q}(\t) + O(B(n)n^{-1/2+\e}\ln^2 n), & k = 0 \\
\widetilde{q}(\t)p(\t)(1-p(\t))^{k-1} + O(B(n)n^{-1/2+\e}\ln^2 n), & k \geq 1.
\end{array}\right.
$$
where $\widetilde{q}(\t) = q(\t)$ for all $\t\notin (-\d, \d) \cup (1-\d, 1 + \d)$.
\end{lemma}
Before proving the lemma, we show how it is used to finish the proof of Theorem \ref{degthm}. Consider the graph $\G_0 \in  \mG(\emptyset, \emptyset)$ in which no assignments or reveals have been made. We expose $(e_1, 2)$ to find that $|E((e_1, 2), \G_0)|$ is asymptotically $G(p(\t), \widetilde{q}(\t))$ distributed. Exposing $(e_1, 2)$ gives a partially generated graph $\G_1 \in \mG(A_1, R_1)$ where $A_1$ is the set of edges assigned while exposing $(e_1, 2)$ and $R_1$ consists of $(e_1, 2)$ and $(e, 1)$ for all $e\in A_1$. By Lemma \ref{CMJbirthlem} we have $|A_1|, |R_1| = O(B(n)) = o(n^{1/2})$ whp. Apply Lemma \ref{exposelemma} to $\G_1$ to find that $|E((e_2, 2), \G_1)|$ is asymptotically $G(p(\t), \widetilde{q}(\t))$ distributed, and consider $\G_2 \in \mG(A_2, R_2)$, where $A_2\setminus A_1$ and $R_2\setminus R_1$ consist of the edges assigned and revealed when exposing $(e_2, 2)$. Repeating this $m$ times keeps the sets $A_i, R_i$ of size $o(n^{1/2})$, and we find that $|E((e_i, 2), \G_{i-1})|$ is asymptotically $G(p(\t), \widetilde{q}(\t))$ distributed for $i = 1,2,\dots, m$. Then
$$
d(n, v) = \sum_{i=1}^m |E((e_i, 2), \G_{i-1})|
$$
and the theorem follows.

The above assumes that each vertex makes its $m$ random choices with replacement. In the process of determining $d(n, v)$, $O(B(n))$ edges are revealed. The probability for any edge $e$ to be adjacent to $v$ is $O(B(n)/n)$, and it follows that the probability that two edges $e_1, e_2$ with $\lceil e_1 / m\rceil = \lceil e_2 / m\rceil$ are adjacent to $v$ is $O(B(n)^2 / n) = o(B(n)n^{-1/2})$. This shows that $d(n, v)$ has the same asymptotic distribution when choices are made with or without replacement.

\subsection{Proof of Lemma \ref{exposelemma}}

For $i \geq 1$ write $X_{si}' = \log_\g(e_{si} / e_{s(i-1)})$, where we say $e_{s0} = e_s$. We will show that the collection $\{X_s' : s\in S\}$ can be coupled to a collection $\{X_s : s\in S\}$ of independent Exp$(\a)$ variables in such a way that $X_s' = X_s + O(n^{-1/2+\e}\ln n)$ for all $s$ with high probability. The lemma will then follow from arguing that a CMJ process on $[0, \t]$ with $\t\leq \log_\g n$ is robust with high probability, in the sense that changing interarrival times by $O(n^{-1/2+\e}\ln n)$ does not change the value of $d(\t)$.

The set of edges $e$ with $e_0 \in E_e^\s$ is $\{e_0 + i, e_0 + i + 1,\dots, e_0'\}$ for some $i\in [m]$ and some $e_0'$. Since $\s$ is $\om$-concentrated, there exists a constant $C > 0$ such that for all edges $e \geq mn/\om$, the largest edge that may choose $e$ is $e'$ where $e(\g - Cn^{-1/2}\ln^2 n) < e' < e(\g + Cn^{-1/2}\ln^2 n)$. Fix such a $C$ and let $\d_1 = Cn^{-1/2}\ln^2 n$, and let $\d = O(n^{-1/2}\ln^2 n)$ be such that $1 - \d < \log_\g (\g + \d_1) < 1 + \d$. Let $\t = \log_\g (pmn / e_0)$. If $\t \leq -\d$ then $e_0\notin \G$, if $\d \leq \t \leq 1-\d$ then $e_0 \in E_n$, and if $\t \geq 1 + \d$ then $e_0\in \G$ but $e_0\notin E_n$. We will be assuming that $\t \notin (-\d, \d) \cup (1 - \d, 1 + \d)$, and leave the cases $\t\in (-\d, \d)$ and $\t \in (1-\d, 1 + \d)$ until the end of the proof.

Now, consider the first edge $e_{01}$ that chooses $(e_0, j_0)$, taken to be $\infty$ if no edge chooses $(e_0, j_0)$. Since $\s$ is $\om$-concentrated and $|R| = O(n^{1/2+\e}\log^k n)$ for some $k\geq 1$, we have $|\widetilde{\OM}(e)| = 2\m e/pm + O(n^{1/2}\ln n) - O(n^{1/2+\e}\log^k n) = 2\m e/pm + O(n^{1/2+\e}\log^k n)$ for all $e > mn / \om$. Since $e_0 > mn / \om$, if $(e_0, j_0) \in \widetilde{\OM}(e)$ then
$$
\Prob{e \mbox{ chooses } (e_0, j_0)} = \frac{pm}{2\m e} + O(n^{-3/2 + \e}\ln^k n),
$$
independently of the random choice of all other edges. Let $i\in [m]$ be the smallest number for which $e_0 \in E_{e_0 + i}^\s$, and suppose $y > 1$ is such that $e_0 \in E_{\lfloor ye_0\rfloor}^\s$. Then if $x = \log_\g y$,
\begin{align}
\Prob{e_{01} > ye_0} & = \prod_{\substack{e = e_0 + i \\ e\notin A}}^{\lfloor ye_0\rfloor} \left(1 - \frac{pm}{2\m e} + O(n^{-3/2 + \e}\ln^k n)\right) \nonumber \\
& = \exp\left\{-\frac{pm}{2\m}\sum_{\substack{e = e_0 + i \\ e\notin A}}^{\lfloor ye_0\rfloor} \left(\frac1e + O(n^{-3/2+\e}\ln^k n)\right)\right\} \nonumber \\
& = \exp\left\{-\a x\left(1 + O\left(\frac{|A|}{n} + n^{-1/2 + \e}\ln^k n\right)\right)\right\} \nonumber \\
& = \exp\left\{-\a x\left(1 + O(n^{-1/2 + \e}\ln^k n)\right)\right\}. \label{nearexp}
\end{align}
This suggests that $X_{01}' = \log_\g(e_{01} / e_0)$ is approximately exponentially distributed, in the range of $y$ for which $e_0 \in E_{\lfloor ye_0\rfloor}^\s$. We will couple $X_{01}'$ to an exponentially distributed random variable, and the coupling technique will depend on whether or not $e_0 \in E_n$. Define $\t = \log_\g(pmn / e_0)$. As noted above, $\t > 1 + \d$ implies $e_0\in \G$ and $e_0\notin E_n$, while $\d < \t < 1-\d$ implies $e_0\in E_n$.

{\bf Case 1, $e_0\notin E_n$.}\\
Suppose $\t > 1 + \d$, so that $e_0\notin E_n$ under our choice of $\s$. By choice of $\d_1$, there exists a $y \in (\g-\d_1, \g+\d_1)$ such that $ye_0$ is the largest edge for which $e_0 \in E_{ye_0}^\s$. Applying \eqref{nearexp} with this $y$, we have $\Prob{e_{01} = \infty} = \exp\left\{-\a(1 + O(n^{-1/2}\ln^k n))\right\}$, since $\log_\g (\g + \d) = 1 + O(n^{-1/2}\ln^2 n)$. For $L > 0$ we define a distribution Exp$(\a, L)$ by saying that $X\sim \mbox{Exp}(\a, L)$ if $\Prob{X > x} = e^{-\a x}$ for $0 < x < L$ and $\Prob{X = \infty} = e^{-\a L}$. We will couple $X_{01}'$ to an Exp$(\a, 1)$ variable, as described below.

Condition on $e_{01}$ and consider $e_{02}$, the second edge that chooses $e_0$. Repeating \eqref{nearexp} shows that $\Prob{e_{02} > ye_{01}} = \exp\{-\a x(1 + O(n^{-1/2+\e}\ln^k n))\}$ where $x = \log_\g y$, for all $y$ such that $e_0 \in E_{ye_{01}}^\s$. The largest such $y$ is $\g e_0 / e_{01} + O(n^{-1/2}\ln^2 n)$, and
$$
\log_\g\left(\frac{\g e_0}{e_{01}} + O(n^{-1/2}\ln^2 n)\right) = 1 - X_{01}' + O(n^{-1/2}\ln^2 n).
$$
We will couple $X_{02}'$ to an Exp$(\a, 1 - X_{01}')$ variable. In general, $X_{0i}'$ will be coupled to an Exp$(\a, 1 - X_{01}' - \dots - X_{0(i-1)}')$ variable, conditioning on $X_{01}',\dots,X_{0(i-1)}'$.

{\bf Case 2, $e_0 \in E_n$.}\\
In the case $\d < \t < 1 - \d$, where $e_0 \in E_n$, we instead couple $X_{01}'$ to an Exp$(\a, \t)$ variable, since the largest edge that may choose $(e_0, j_0)$ is $m\n_n = pmn + O(n^{-1/2}\ln n)$, the largest edge in $\G$. We will couple $X_{0i}'$ to an Exp$(\a, \t - X_{01}' - \dots - X_{0(i-1)}')$ variable.

{\bf Coupling the variables:} Let $\t' = \min\{1, \t\}$. In terms of Exp$(\a, L)$ variables, we can define a Poisson process on $[0, \t']$ as follows. Let $X_{01} \sim \mbox{Exp}(\a, \t')$. Conditioning on $X_{01} = x_{01} < 1$ we define $X_{02} \sim \mbox{Exp}(\a, \t'-x_{01})$. In general let $X_{0i} \sim \mbox{Exp}(\a, \t'-x_{01}-\dots-x_{0(i-1)})$ until $X_{0k} = \infty$. Then $X_{01},\dots,X_{0(k-1)}$ are the interarrival times for a Poisson process of rate $\a$ on $[0, 1]$.

We will now describe the coupling explicitly. Let $U_{01}, U_{02},\dots$ be a sequence of independent uniform $[0,1]$ variables. The variable $X_{01}\sim \mbox{Exp}(\a, 1)$ is given by
$$
X_{01} = \left\{\begin{array}{ll}
-\a^{-1}\ln U_{01}, & e^{-\a} < U_{01} < 1, \\
\infty, & 0 < U_{01} < e^{-\a}.
\end{array}\right.
$$
and for $i \geq 1$, conditioning on $X_{01} = x_{01}, \dots, X_{0i} = x_{0i}$ where $x_{01} + \dots +  x_{0i} < 1$,
$$
X_{0(i+1)} = \left\{\begin{array}{ll}
-\a^{-1}\ln U_{0(i+1)}, & e^{-\a(1 - x_{01} - \dots - x_{0i})} < U_{0(i+1)} < 1, \\
\infty, & 0 < U_{0(i+1)} < e^{-\a(1-x_{01} - \dots - x_{0i})}.
\end{array}\right.
$$
Define $X_{01}' = \min\{\log_\g y : \Prob{e_{01} > ye_0} \leq U_{01}\}$, taken to be $\infty$ if the set is empty. Recall that $\d_1 = O(n^{-1/2}\ln^2 n)$ is such that $1-\d_1 < \log_\g (\g + \d) < 1 + \d_1$. Then by \eqref{nearexp} and the choice of $\d$,
$$
\mbox{if } \ U_{01} > e^{-\a(1 - \d_1)} \ \mbox{ then } \ X_{01}' = \frac{-1}{\a + O(n^{-1/2 + \e}\ln^k n)}\ln U_{01} = X_{01} + O(n^{-1/2+\e}\ln^k n),
$$
and if $U_{01} < \exp\{-\a(1 + \d_1)\}$ then $X_{01}' = \infty$. Say that this coupling of $X_{01}, X_{01}'$ is {\em good} if either $X_{01}, X_{01}'$ are both infinite, or $X_{01}' = X_{01} + O(n^{-1/2 + \e}\ln^k n)$, and {\em bad} otherwise. The above shows that
\begin{align*}
\Prob{\mbox{the coupling of $X_{01}, X_{01}'$ is bad}} 
& = \Prob{e^{-\a(1 + \d_1)} < U_{01} < e^{-\a(1- \d_1)}} \\
& = O(n^{-1/2}\ln^2 n).
\end{align*}
Suppose $i > 1$ and condition on the couplings of $X_{0j}, X_{0j}'$ being good with $X_{0j}, X_{0j}' < \infty$ for all $1\leq j < i$. Define $X_{0i}' = \min\left\{\log_\g y : \Prob{e_{0i} > ye_{0(i-1)} \middle| X_{01}',\dots, X_{0(i-1)}'} \leq U_{0i}\right\}$. We repeat the above argument to show that the coupling of $X_{0i}', X_{0i}$, conditioning on previous couplings, is bad with probability $O(n^{-1/2}\ln^2 n)$.

Let $\mathcal{C}_0$ be the event that the coupling of the $X_{0i}$ making up the Poisson process $\mP_0$ is good for all $i$. Since the process has $O(\log n)$ arrivals with probability $1 - o(n^{-1})$, we have $\Prob{\mathcal{C}_0} = 1 - O(n^{-1/2+\e}\ln^3 n)$. After revealing $e_0$, the partially generated graph $\widetilde{\G}$ is in $\mG(A', R')$, where $|A' \setminus A| = O(\log n)$ whp and $|R' \setminus R| = 1$. Thus, the coupling argument can be applied to $O(n^{1/2+\e})$ Poisson processes with $|A|, |R| = O(n^{1/2 + \e}\ln^k n)$ being maintained.

Let $S'(t)$ be the state of the set $S'$ after Step 2 of the algorithm has been executed $t$ times, and let $S_c'(t)$ be the corresponding set in the CMJ generating algorithm of Section \ref{poissonsec}. We just showed that $S'(1) = S_c'(1)$ with probability $1 - O(n^{-1/2+\e}\ln^2 n)$. For any process $\mP_s$ that appears, we apply a coupling using the technique above, and we have $S'(t) = S_c'(t)$ for all $1\leq t < B(n)$ with probability $1 - O(B(n)n^{-1/2 + \e}\ln^2 n) = 1-o(1)$. We also have $S_c'(B(n)) = \emptyset$ with probability $1 - o(n^{-1})$, by Lemma \ref{CMJbirthlem}, so
$$
\Prob{S'(t) \neq S_c'(t) \ \mbox{for some $t \geq 1$}} = o(1).
$$
Condition on the two algorithms producing the same set $S$ of strings. For $s = 0s_1\dots s_j\in S$ we have
$$
\t_s = \sum_{i = 1}^{s_1} X_{0i} + \sum_{i=1}^{s_2} X_{0s_1i} + \dots + \sum_{i=1}^{s_j} X_{0s_1\dots s_{j-1}i},
$$
and the same identity holds with $\t_s, X_r$ replaced by $\t_s', X_r'$. If $s = 0s_1\dots s_j$ let $|s| = j$ be the {\em generation} of $s$. With probability $1-o(n^{-1})$, each Poisson process has $O(\log n)$ arrivals, so each $s_i = O(\log n)$. Thus $\t_s$ is a sum of $O(|s|\log n)$ variables $X_r$, and if all couplings are good then $\t_s' = \t_s + O(|s|n^{-1/2}\ln^{k+1} n)$ for all $s\in S$. We need to bound $|s|$.

{\bf Claim:} Consider a CMJ process with rate $\a > 0$ and lifetime $1$. Let $0\leq \t\leq\log_\g n$ and $S(\t) = \max\{|s| : \t_s \leq \t\}$. Then $\Prob{S(\t) > \log_\g^2 n} = o(n^{-1})$.

{\bf Proof of claim:} Let $P_k(\t)$ denote the number of processes $\mP_s$ with $|s| = k$ and $\t_s < \t$. Condition on $\mP_0$ having arrivals at time $x_1,\dots,x_\ell$. Then $\mC$ can be seen as $\mP_0$ together with $\ell$ independent CMJ processes $\mC^1,\dots,\mC^\ell$ on $[x_1,\t], \dots, [x_\ell, \t]$ respectively. Then
$$
P_k(\t) = \sum_{j=1}^\ell P_{k-1}^j(\t - x_j)
$$
where $P_{k-1}^j(\t-x_j)$ counts the number of $(k-1)$th generation processes started before $\t-x_j$ in $\mC^j$. Let $U$ denote a uniform $[0,1]$ random variable. Removing the conditioning and taking expectations, we have
$$
\E{P_k(\t)} = \sum_{\ell\geq 0} \frac{e^{-\a}\a^\ell}{\ell!}\sum_{j=1}^\ell \E{P_{k-1}^j(\t - U)} = \E{P_{k-1}(\t-U)}\sum_{\ell\geq 1}\frac{e^{-\a}\a^\ell}{(\ell-1)!} = \a\E{P_{k-1}(\t-U)}.
$$
Here we use the fact that if we condition on a Poisson process on $[0,1]$ having $\ell$ arrivals, the arrival times are independently uniformly distributed. Note that $P_k(\t) = 0$ for all $\t < 0$.

We show by induction over $k$ that $P_k(\t) \leq (\a\t)^k / k!$ for all integers $k\geq 0$ and all $\t \geq 0$. For the base case we have $P_0(\t) = 1$ for $\t\geq 0$. If $P_k(\t) \leq (\a\t)^k/k!$ for all $\t \geq 0$ then if $\t \geq 0$,
\begin{align*}
& \E{P_{k+1}(\t-U)} = \a\E{\E{P_k(\t-U)\mid U}} \\
&  \leq \a\E{\frac{\a^k(\t-U)^k}{k!}} = \frac{\a^{k+1}}{(k+1)!}(\t^{k + 1} - \max\{0, \t-1\}^{k+1}) \leq \frac{(\a\t)^{k+1}}{(k+1)!},
\end{align*}
where we use the fact that $P_k(\t) = 0$ for $\t < 0$, and the induction is complete.

Let $k = \log_\g^2 n$. Then by Markov's inequality and the bound $k! \geq (k/e)^k$,
$$
\Prob{\exists s : \t_s < \t \mbox{ and } |s| = k} \leq \E{P_k(\t)} \leq \bfrac{e\a\t}{k}^k \leq \bfrac{e\a}{\log_\g n}^{\log_\g n} = O(n^{-C})
$$
for any $C > 0$. Since $P_{k'}(\t) \leq P_k(\t)$ for any $k' \geq k$, the claim follows.\\
{\bf End of proof of claim.}

We have shown that with high probability, the graph algorithm produces a set $S$ and a set $\{\t_s' : s\in S\}$ that matches the set $\{\t_s : s\in S\}$ of a CMJ process in the sense that $\t_s' = \t_s + O(n^{-1/2 + \e}\ln^{k+3} n)$ for all $s$. Since $d(\t)$ counts the number of $\t_s$ in the interval $(\t-1,\t)$, we can finish the proof by arguing that
\begin{equation}\label{dtid}
\{s\in S: \t - 1 < \t_s < \t\} = \{s\in S: e_s \in E_n\}.
\end{equation}
Since $\s$ is $\om$-concentrated, every edge $e\in E_n$ satisfies $\log_\g(e / e_0) \in (\t - 1 - O(n^{-1/2}\ln n), \t + O(n^{-1/2}\ln n))$ where $\t = \log_\g (pmn / e_0)$. Condition on $\t_s' = \t_s + O(n^{-1/2+\e}\ln^{k+3} n)$ for all $s\in S$. If \eqref{dtid} is false, there must exist some $s\in S$ such that either $\t_s = \t-1 + O(n^{-1/2+\e}\ln^{k+3} n)$ or $\t_s = \t + O(n^{-1/2+\e}\ln^{k+3} n)$. The probability of this is $O(B(n)n^{-1/2}\ln^{k+3} n)$, since a CMJ process with at most $B(n)$ active processes locally behaves like a Poisson process with rate at most $\a B(n)$. This finishes the proof of the lemma for $\t\notin (-\d, \d) \cup (1-\d, 1 +\d)$.

If $\t\in (-\d, \d)\cup (1-\d, 1+\d)$, then \eqref{dtid} is false with some significant probability, since one set may contain $s = 0$ while the other one does not. The function $\widetilde{q}(\t)$ accounts for this event.

\section{The degree sequence}\label{degseqsec}

For $k\geq 0$ let $X_k(n)$ denote the number of vertices of degree $k$ in $G_n$. In this section we prove the following. Recall $\eta = -\ln \g / \ln \zeta > 2$, defined when $\a > 1$, see Section \ref{confun}. Recall that $\mathcal{D}$ denotes the event that at some point, the graph process contains no edges. The probability of $\mathcal{D}$ depends on the initial graph $H = G_{t_0}$, see Lemma \ref{sigmalemma}.
\begin{theorem}\label{degseqthm}
Condition on $\overline{\mathcal{D}}$. There exists a sequence $\{x_k : k\geq 0\}$ such that
\begin{enumerate}[(i)]
\item if $\a < 1$ then $x_k = \a^{k(1+o_k(1))}$ and if $\a > 1$ then there exist constants $a, b > 0$ such that $x_k = ak^{-\eta-1} + O_k(k^{-\eta-2}\log^b k)$, and
\item for any fixed $k \geq 0$, $X_k(n) = x_kn(1 + o_n(1))$ with high probability as $n\to \infty$.
\end{enumerate}
\end{theorem}

\begin{proof}
Fix $k\geq 0$. We begin by showing that $X_k(n) = (1+o_n(1)) \E{X_k(n)}$ whp. We will use Azuma's inequality in the general exposure martingale setting in \cite[Section 7.4]{alonspencer}. To do this, fix a feasible $\s$ and consider the master graph $\G = \G_n^\s(H)$ for a fixed starting graph $H$ (see Section \ref{modelsec}). Let $\G_0$ be the unexplored graph, as defined in Section \ref{mastersec}, and define a sequence $\G_1, \G_2,\dots,\G_M = \G$ of partially generated graphs. Here $\G_{i}$ is obtained from $\G_{i-1}$ by letting edge $m\nu_H + i$ make its random choice. Consider the edge exposure martingale $Y_i^\s = \E{X_k(n) \mid \G_i}$. If $E_n$ denotes the edge set of $G_n$, then $X_k(n)$ is given as a function of $\G$ by counting the number of vertices which are incident to exactly $k$ edges of $E_n$. This martingale satisfies the Lipschitz inequality $|Y_i^\s - Y_{i-1}^\s| \leq 3$, since the degrees of at most $3$ vertices are affected by changing the choice of one edge (see e.g. Theorem 7.4.1 of \cite{alonspencer}). By Azuma's inequality, conditioning on $\s$ we have $|X_k(n) - \E{X_k(n)}| < n^{1/2}\ln n$ with probability $1 - e^{-\OM(\ln^2 n)}$, noting that $M = m(\nu_n - \nu_H)$ is of order $n$. We will show that $\E{X_k(n)} / n$ has essentially the same limit for all feasible and $\om$-concentrated $\s$, setting $\om=\log\log n$, and the result will follow since $\s$ is $\om$-concentrated with probability $1-o(n^{-1})$ (Lemma \ref{sigmalemma}) and $X_k(n)\leq n$. Fix a feasible and $\om$-concentrated $\s$ for the remainder of the proof.

Recall that $G^m(p, q)$ denotes the distribution of the sum of $m$ independent $G(p, q)$ variables. If $X \sim G^m(p, q)$ and $k \geq m$ then
\begin{align}
\Prob{X = k} & = \sum_{\ell = 1}^m \binom{m}{\ell}\sum_{\substack{k_1 + \dots + k_\ell = k \\ k_1, \dots, k_\ell > 0}} (1-q)^{m-\ell}\prod_{i = 1}^\ell qp(1-p)^{k_i - 1} \nonumber \\
& = \sum_{\ell = 1}^m \binom{m}{\ell} \binom{k-1}{\ell-1}(1-q)^{m-\ell}q^\ell p^\ell (1-p)^{k-\ell}. \label{mdist}
\end{align}
Here $\ell$ represents the number of nonzero terms in the sum $X = X_1+\dots+X_m$, and $\binom{k-1}{\ell-1}$ is the number of ways to write $k$ as a sum of $\ell$ positive integers. By linearity of expectation,
$$
\E{X_k(n)} = \sum_{v = 1}^{n} \Prob{d(n, v) = k}.
$$
Let $\om = \log\log n$. By Theorem \ref{degthm} we have
$$
\sum_{v=1}^n \Prob{d(n, v) = k} = O\bfrac{n}{\om} +\sum_{v=n/\om}^n \left(\Prob{G^m(p(\t), \widetilde{q}(\t)) = k} + O\left(\frac{B(n)\ln^3 n}{n^{1/2}}\right)\right).
$$
Summing the $O(n^{-1/2}B(n)\ln^3 n)$ terms gives a cumulative error of $O(n^{1/2}B(n)\ln^3 n) = o(n)$, since either $B(n) = O(\log n)$ or $B(n) = O(n^{1/\eta}\ln n)$ (see Section \ref{poissonsec}) and $\eta > 2$ (see Lemma \ref{constlem} (\ref{etalemma})). So if $k \geq m$ and $\t_v = \log_\g(pn / v)$, by \eqref{mdist},
\begin{equation}\label{Exkn1}
\E{X_k(n)} = O\left(\frac{n}{\om}\right) + \sum_{v=n/\om}^n\sum_{\ell = 1}^m \binom{m}{\ell}\binom{k-1}{\ell-1}(1-\widetilde{q}(\t_v))^{m-\ell}\widetilde{q}(\t_v)^\ell p(\t_v)^\ell(1-p(\t_v))^{k-\ell},
\end{equation}
where $\widetilde{q}(\t) \in (0,1)$ and $\widetilde{q}(\t) = q(\t)$ outside $(-\d, \d) \cup (1-\d, 1 + \d)$ for some $\d = O(n^{-1/2}\ln n)$. For $n/\om \leq v\leq n$ we have $\log_\g p \leq \t_v \leq \log_\g (p\om)$ (note that $\log_\g p < 0$), and for any $\t$ in the interval, the number of $v$ for which $\t \leq \t_v \leq \t + \e$ is $pn\e\ln (\g)\g^{-\t} + O(\e^2)$. Viewing the sum as a Riemann sum, we have
\begin{align}
& \lim_{n\to\infty} \frac{1}{n} \sum_{v=n/\om}^n (1-\widetilde{q}(\t_v))^{m-\ell}\widetilde{q}(\t_v)^\ell p(\t_v)^\ell(1-p(\t_v))^{k-\ell} \nonumber \\
= & p\ln \g\int_{\log_\g p}^\infty \frac{(1-\widetilde{q}(\t))^{m-\ell}\widetilde{q}(\t)^\ell p(\t)^\ell (1-p(\t))^{k-\ell}}{\g^\t}d\t \nonumber \\
= & O(n^{-1/2}\ln n) + p\ln \g \int_0^\infty \frac{(1-q(\t))^{m-\ell}q(\t)^\ell p(\t)^\ell(1-p(\t))^{k-\ell}}{\g^\t}d\t. \label{integral}
\end{align}
The last identity comes from (i) the fact that $\widetilde{q}(\t) = q(\t)$ outside a set of total length $O(n^{-1/2}\ln n)$, (ii) the fact that the integral converges since the integrand is dominated by $\g^{-\t}$ where $\g > 1$, and (iii) the fact that $q(\t) = 0$ for $\t < 0$.

Plugging \eqref{integral} into \eqref{Exkn1} we have
\begin{equation}\label{Exkn}
\lim_{n\to\infty}\frac{\E{X_k(n)}}{n} = p\ln\g\sum_{\ell=1}^m \binom{m}{\ell}\binom{k-1}{\ell-1}\int_0^\infty \frac{(1-q(\t))^{m-\ell}q(\t)^\ell p(\t)^\ell (1-p(\t))^{k-\ell}}{\g^\t} d\t.
\end{equation}
Let
\begin{equation}\label{felldef}
f_\ell(\t) = \frac{(1-q(\t))^{m-\ell} q(\t)^\ell p(\t)^\ell (1-p(\t))^{k-\ell}}{\g^\t}.
\end{equation}
Our aim is to calculate $\int_0^\infty f_\ell(\t) d\t$.

{\bf Case 1: $\a > 1$.}\\
By Lemma \ref{constlem} (\ref{largealpha}) we have $p(\t) \geq \l_3 \zeta^\t$ for all $\t \geq 0$, where $\l_3 > 0$. Let $\psi(k) = -\log_\zeta ((k-\ell)/(C\ln k))$ for some constant $C > 0$, noting that $\psi(k)\to \infty$ when $k\to \infty$. Making $C$ large enough,
\begin{equation}\label{integralsmalltau}
\int_0^{\psi(k)} f_\ell(\t) \leq \psi(k)(1-\l_3 \zeta^{\psi(k)})^{k-\ell} \leq \psi(k) e^{-\l_3C\ln k} = O(k^{-\eta - 2}). 
\end{equation}
Here we used the fact that $f_\ell(\t) \leq (1-p(\t))^{k-\ell}$.

Again by Lemma \ref{constlem} (\ref{largealpha}) we have $p(\t) = \l_3 \zeta^\t + O(\zeta^{2\t})$ and $q(\t) = 1 - \zeta + O(\zeta^\t)$. Suppose $\t \geq \psi(k)$. Then $k\zeta^{2\t} = o_k(1)$ and
$$
f_\ell(\t) = \frac{\zeta^{m-\ell}(1-\zeta)^{\ell}(\l_3\zeta^\t)^\ell(1-\l_3\zeta^\t)^{k-\ell}}{\g^\t}\left(1 + O(m\zeta^{\t}) + O_k(k\zeta^{2\t})\right).
$$
Indeed, each of the $m$ factors involving $q(\t)$ contributes an error factor of $1+O(\zeta^\t)$ and each of the $k$ factors involving $p(\t)$ contributes an error factor of $1+O(\zeta^{2\t})$. We have $m\zeta^\t = O(\ln k / k)$ and $k\zeta^{2\t} = O(\ln^2 k / k)$, so
\begin{equation}\label{fellapprox}
f_\ell(\t) = \frac{\zeta^{\t\ell}(1-\l_3\zeta^{\t})^{k-\ell}}{\g^\t}\left(\l_3^\ell \zeta^{m-\ell}(1-\zeta)^\ell+O\bfrac{\ln^2 k}{k}\right).
\end{equation}
Note that $\l_3, \zeta, m$ and $\ell$ are independent of $k$ and $\t$.

{\bf Claim:} If $\a > 1$ there exists a constant $c_\ell$ such that
$$
\int_{\psi(k)}^\infty \frac{\zeta^{\t\ell}(1-\l_3\zeta^\t)^{k-\ell}}{\g^\t} = c_\ell k^{-\eta - \ell}  + O(k^{-\eta-\ell-1}).
$$

It will follow from the claim and \eqref{fellapprox} that for some constant $c_\ell'$,
\begin{equation}\label{integrallargetau}
\int_{\psi(k)}^\infty f_\ell(\t) d\t = c_\ell' k^{-\eta-\ell}\left(1 + O\bfrac{\ln^2 k}{k}\right).
\end{equation}

{\bf Proof of claim:}
We make the integral substitution $u = \l_3\zeta^\t$, noting that $\t = \log_\zeta (u/\l_3)$ so (recalling that $\eta = -\ln\g / \ln\zeta$, see Section \ref{confun})
$$
\g^{-\t} = \exp\left\{-\ln \g \frac{\ln (u / \l_3)}{\ln \zeta}\right\} = \bfrac{u}{\l_3}^\eta.
$$
This implies that the integral equals (up to a multiplicative constant)
\begin{align*}
\int_0^{\frac{C\l_3 \ln k}{k-\ell}} u^{\eta + \ell - 1}(1-u)^{k - \ell} du & = \int_0^1 u^{\eta+\ell-1}(1-u)^{k-\ell}du - \int_{\frac{C\l_3\ln k}{k-\ell}}^1 u^{\eta + \ell - 1}(1-u)^{k - \ell}du \\
& = B(\eta +\ell, k - \ell + 1) + O\left(k^{-C\l_3}\right)
\end{align*}
where $B(x, y) = \int_0^1 u^{x-1}(1-u)^{y-1}du$ denotes the Beta function. Here the $O(k^{-C\l_3})$ term comes from bounding $u^{\eta + \ell - 1} \leq 1$ and $1 - u \leq e^{-C\l_3\ln k / (k-\ell)}$. Taking $C$ to be large enough makes the error $O(k^{- \eta - m - 1})$ (recall that $\ell \leq m$). As $k\to \infty$, Stirling's formula provides an asymptotic expression for $B(\eta, k + 1)$:
$$
B(\eta + \ell, k -\ell + 1) = \G(\eta + \ell) k^{-\eta - \ell} + O(k^{-\eta-\ell-1}),
$$
where $\G$ denotes the Gamma function.
{\bf End of proof of claim.}

We finish the proof for $\a > 1$ by noting that by Stirling's formula, for some constant $s_\ell$
\begin{equation}\label{binom}
\binom{k-1}{\ell - 1} = s_\ell k^{\ell - 1} + O(k^{\ell - 2})
\end{equation}
Plugging \eqref{integralsmalltau}, \eqref{integrallargetau} and \eqref{binom} into \eqref{Exkn} shows that
\begin{align*}
\frac{\E{X_k(n)}}{n} & \to p\ln \g\sum_{\ell=1}^m \binom{m}{\ell}\binom{k-1}{\ell-1} \int_0^\infty f_\ell(\t) dt \\
& =  p\ln\g\sum_{\ell=1}^m \binom{m}{\ell}(s_\ell k^{\ell-1} + O(k^{\ell - 2})) (c_\ell' k^{-\eta-\ell} + O(k^{-\eta-\ell-1}\ln^2 k))  \\
& = \left(p\ln\g \sum_{\ell = 1}^m \binom{m}{\ell} s_\ell c_\ell'\right)k^{-\eta-1} + O(k^{-\eta-2}\ln^{\eta+m+3} k).
\end{align*}
Here the expression in brackets depends only on $p, m$, and this is the constant $a$ in the statement of the theorem.

{\bf Case 2:} $\a < 1$.\\
In this case we need not be as careful. By Lemma \ref{constlem} (\ref{smallalpha}) we have $1 - p(\t) = \a - \l_1/\zeta^\t + O(\zeta^{2\t})$ where $0 < \l_1 < \a$ and $\zeta > 1$, so we can write
$$
f_\ell(\t) = \a^{k-\ell} \frac{(1-q(\t))^{m - \ell} q(\t)^\ell p(\t)^{\ell}\left(\frac{1 - p(\t)}{\a}\right)^{k-\ell}}{\g^\t}
$$
and the calculation of $\a^{-(k-\ell)}\int_0^\infty f_\ell(\t) d\t$ proceeds much like the $\a > 1$ case. We find that
$$
\int_0^\infty f_\ell(\t)d\t = \a^{k-\ell}O(k^C) = \a^{k(1 + o_k(1))}
$$
for some constant $C > 0$. Summing over $\ell = 1,\dots, m$ does not affect this expression.
\end{proof}

\section{The largest component}\label{giantsec}

This section deals with connectivity properties of $G_n$. Note that $G_n$ is disconnected whp since one can show that the number of isolated vertices is $\Omega(n)$ whp. It is also the case that the set of non-isolated vertices is disconnected whp, since the probability that a vertex $v$ shares a component only with its $m$ older neighbors is a nonzero constant, as can be seen by methods similar to those used in the proof of Lemma \ref{Xlemma} below.

In the following theorem, the size of a component refers to the number of vertices in the component. Recall that $B(n) = \l\ln n$ if $\a < 1$ and $B(n) = \l n^{1/\eta}\ln n$ if $\a > 1$, for a constant $\l > 0$. Recall also that $\mathcal{D}$ denotes the event that the graph process contains zero edges at some point (see Lemma \ref{sigmalemma}).

Note that the number of vertices in $G_n$ is $pn + O(n^{1/2}\ln n)$ whp, so when $m\geq 2$ and $\a > 1$, Theorem \ref{giantthm} states that whp the number of vertices outside the giant component is $O_m(c^m n)$ for some $0 < c < 1$.
\begin{theorem}\label{giantthm}
Condition on $\overline{\mathcal{D}}$.
\begin{enumerate}[(i)]
\item There exists a $\xi = \xi(m, p) \in (0, p)$ such that the number of isolated vertices in $G_n$ is $\xi n(1 + o_n(1))$ whp. If $\a > 1$ then $\xi = O_m(c^m)$ for some $0<c<1$.
\item If $m = 1$, all components in $G_n$ have size $O(\D\log n)$ whp, where $\D$ denotes the maximum degree of $G_n$.
\item If $m \geq 2$, whp there exists a component containing at least $p(1-\xi)(1-(13/14)^{m-1})n$ vertices while all other components have size $O(\log n)$. 
\end{enumerate}
\end{theorem}

The remainder of the section is devoted to the proof of this theorem. Let $\om = \log \log n$. We fix a feasible and $\om$-concentrated $\s$, see Lemma \ref{sigmalemma}. We also fix $\e > 0$ with $1/2-\e > 1/\eta$ if $\a > 1$ and $\e < 1/2$ if $\a < 1$.

We first prove (i). The existence of $\xi$ is provided by Theorem \ref{degseqthm} (ii), so we need only prove that $\xi = O_m(c^m)$ for some $0<c<1$ when $\a > 1$. Fix a vertex $v\geq n/\om$. By Theorem \ref{degthm} the probability for $v$ to be isolated is $(1-q(\t))^m$ for some $\t$. By Lemma \ref{constlem} (\ref{largealpha}), $\a > 1$ implies $1 - q(\t) \leq \zeta < 1$ for all $\t$, so the probability of being isolated is at most $\zeta^m$. By linearity of expectation we expect at most $\zeta^m pn + O(n/\om)$ vertices to be isolated, accounting for the $n/\om$ vertices for which Theorem \ref{degthm} does not apply. Theorem \ref{degseqthm} shows that the number of vertices of degree zero is within $O(n^{1/2}\ln n)$ of its mean with high probability, so the number of isolated vertices is at most $2p\zeta^m n$ whp. This finishes part (i), and the remainder of the section is devoted to proving (ii), (iii).

The proof will rely heavily on the master graph $\G$ defined in Section \ref{mastersec}. We will define an algorithm that searches for a large connected edge set in $\G$, which remains connected when restricting to the edge set $E_n$ of $G_n$.

Orient each edge $\{u, v\}$ in $\G$ from larger to smaller, i.e. $v\to u$ if $v > u$. Then $d^+(v) = m$ for all $v\geq 1_H$ and $d^+(v) = 0$ for $v < 1_H$. When $m = 1$, this implies that $\G$ is a forest in which each tree is rooted in $\{1,\dots,1_H - 1\}$, and any edge is oriented towards the root in its tree. Restricting to $E_n$ breaks the trees into smaller trees. Let $v \in V_n$. Then there exists a unique vertex $u \notin V_n$ that is reachable from $v$ via directed edges in $E_n$. The connected component of $v$ is $T_u$, where $T_u$ is the tree rooted at $u$ of vertices which can reach $u$ via a directed path. This shows that the connected components in $G_n$ are $\{T_u : u \notin V_n\}$ when $m = 1$.

We now show that $|T_u| = O(d(n, u)\log n)$ for all $u$ whp. Let $u \notin V_n$ and let $v_1,\dots,v_k$ be the neighbors of $u$ in $V_n$, and let $e_i$ be the unique edge oriented out of $v_i$ for $i=1,\dots,k$. Expose $(e_1,2),\dots,(e_k, 2)$. For any edge $e$ found, we expose $(e, 1)$ and $(e, 2)$. Repeating the coupling argument of Lemma \ref{exposelemma} one can show that the of descendants of $e_1$ can modelled by a CMJ process of rate $2\a$. The number of descendants of $e_1$ is geometrically distributed with rate $e^{-2\a\t_1}$ for $\t_1 = \log_\g(pmn / e_1)\leq 1 + O(n^{-1/2}\ln n)$. With high probability each $e_i$ has $O(\log n)$ descendants, and it follows that whp $|T_u| = O(d(n, u)\log n)$ for all $u\notin V_n$. In particular, the largest component has size $O(\D\log n)$ where $\D$ denotes the maximum degree of $G_n$. In this paper we make no attempt to bound $\D$.

Let $m\geq 2$ for the remainder of the section. We now loosely describe the intuition that will help us prove the theorem. Suppose $e_1,\dots,e_m$ are the $m$ edges oriented out of $v\in V_n$ in $G_n$. We imagine splitting $v$ into $m$ smaller vertices $v_1,\dots,v_m$ with $d^+(v_i) = 1$ for each $i$. In Section \ref{distributionsec} we saw that each edge $e$ directed into $v$ can be traced back to a unique $e_i$, in that $e$ either directly chooses $(e_i, 2)$ or chooses $(e', 1)$ for some $e'$ that chooses $(e_i, 2)$, and so on. If $e$ can be traced back to $e_i$, we make it point to $v_i$. Let $G_n'$ be the graph in which all vertices in $V_n$ are split into $m$ parts in this fashion. In $G_n'$ vertices have out-degree $0$ or $1$, and we can define trees $T_u$ as above for $u\notin V_n$. Then each $v\in V_n$ is associated with $m$ trees, namely the $m$ connected components of $v_1,\dots,v_m$ in $G_n'$.

We now make this precise. Let $u\notin V_n$. In Section \ref{distributionsec} we saw how to find the neighbors of $u$ in $V_n$ by exposing $(e_1, 2),\dots,(e_m,2)$ for the $m$ edges $e_1,\dots,e_m$ oriented out of $u$ in $\G$. We start building $T_u$ by letting $u$ be the root, and the children of $u$ each vertex $v\in V_n$ that is adjacent to $u$. For such a $v$, let $e_v$ be an edge that was found when exposing $(e_1,2),\dots,(e_m, 2)$. Expose $(e_v, 2)$ to find all neighbors of $v$ that can be traced back to the edge $e_v$. The children of $v$ in $T_u$ will be all neighbors of $v$ that are incident to some edge that can be traced back to the edge $e_v$. Repeat this for all $v\in V_n$ in $T_u$. Note that $T_u$ may not be a tree, since two edges adjacent to the same vertex may be found when exposing edges.

With this definition, we can partition the edges of $G_n$ into $\{T_u : u\notin V_n\}$. In particular, for each $e\in E_n$ there is a unique vertex $u\notin V_n$ such that $e\in T_u$. Write $T_e = T_u$. The idea behind the algorithm described in detail below is to do a ``breadth-first search on the $T_u$''. Starting with a free edge $x_0\in E_n$, we determine (part of) $T_{x_0}$. For any edge $f\in T_{x_0}$, we expect the other $m-1$ edges oriented out of the same vertex as $f$ to be free. These $m-1$ edges provide the starting point for $m-1$ future rounds of the algorithms, and in each round a new $T_u$ is determined.

For a vertex $v_0$ let $C_\G(v_0), C_G(v_0)$ be the set of edges in the connected component of $v_0$ in $\G, G$, respectively. Starting with a vertex $v_0$ and the graph $\G_0\in \mG(\emptyset,\emptyset)$, we use the following algorithm to find a set $C(v_0) \subseteq C_G(v_0)$. An explanation of the algorithm follows immediately after its description. See Figure \ref{fig:algo} for an example outcome of one round of the algorithm.
\begin{enumerate}
\item[0.] If $v_0\in V_n$ let $C = X = \{m(v_0 - 1) + 1,\dots,mv_0\}$, and $A = R = \emptyset$. If $v_0 \notin V_n$, set $C = X = A = R = \emptyset$ and $Q(x_0) = \{(m(v_0 - 1) + 1, 2), \dots, (mv_0, 2)\}$ and go to step 3.
\item If $X = \emptyset$, stop. If $X\neq\emptyset$ choose an edge $x_0\in X$ and remove it from $X$. Set $Q(x_0) = \{(x_0, 1), (x_0, 2)\}$, $X_1(x_0) = \emptyset$ and $Y_1(x_0) = \emptyset$.
\item Choose $(x_1, j_1) \in \OM(x_0) \setminus R$ uniformly at random.
\begin{enumerate}[(2.1.)]
\item If $x_1 \in A$, do nothing.
\item If $x_1\in E_n$, add $D(x_1)$ to $X_1(x_0)$.
\item If $x_1\notin E_n$ and $j_1 = 2$, add $(x_1, 2)$ to $H$.
\item If $x_1\notin E_n$ and $j_1 = 1$, choose $(x_2, j_2) \in \widetilde{\OM}(x_1)$ uniformly at random. Repeat until one of the following holds:
\begin{enumerate}[(2.4.1.)]
\item $j_1 = j_2 = \dots = j_{k-1} = 1$ and $j_k = 2$. Add $(x_1, 1), \dots, (x_{k-1}, 1)$ and $(x_k', 2)$ for all $x_k'\in D(x_k)$ to $Q(x_0)$. Set $\OM(x_i) = \emptyset$ for $i=1,2,\dots,k-1$.
\item $j_1 = j_2 = \dots = j_{k-1} = 1$ and $x_k < m\n_H$. Let $v$ be the vertex (in $H$) corresponding to $(x_k, j_k)$. Add $(x_1, 1),\dots, (x_{k-1}, 1)$ to $Q(x_0)$, along with $(x', j')$ for all edges $x'$ incident to $v$ in $H$, for the proper choice of $j'$.
\end{enumerate}
\end{enumerate}
Add $x_0, x_1,\dots,x_{k-1}$ to $A$.
\item While $Q(x_0)$ is nonempty, repeat the following.
\begin{enumerate}[(3.1.)]
\item Pick $(h, j)\in Q(x_0)$ and remove it from $Q(x_0)$. Let $Y' = \{(h, j)\}$. Add $h$ to $Y_1(x_0)$. While $Y'\neq\emptyset$ repeat the following:
\begin{enumerate}
\item[(3.1.1)] Choose $(y, i) \in Y'$ and remove it from $Y'$. For each $e\notin X\cup A$ with $(y, i) \in \OM(e)$, query whether $e$ chooses $(y, i)$, i.e. set $\f(e) = (y, i)$ with probability $1 / |\OM(e)|$ and remove $(y, i)$ from $\OM(e)$ otherwise. If $e$ chooses $(y, i)$ then add $(e, 1)$ to $Y'$ and $Y_1(x_0)$, and add all edges $f\neq e$ with $\lceil f/m\rceil = \lceil e/m\rceil$ to $X_1(x_0)$ and $X$. If $e\in E_n$ then also add $(e, 2)$ to $Y'$ and $Y_1(x_0)$.
\end{enumerate}
\end{enumerate}
\item Set $C \leftarrow C \cup X_1(x_0) \cup (Y_1(x_0)\cap E_n)$. Go to step 1.
\end{enumerate}

\begin{figure}
\begin{center}
\includegraphics[width = 0.9\textwidth]{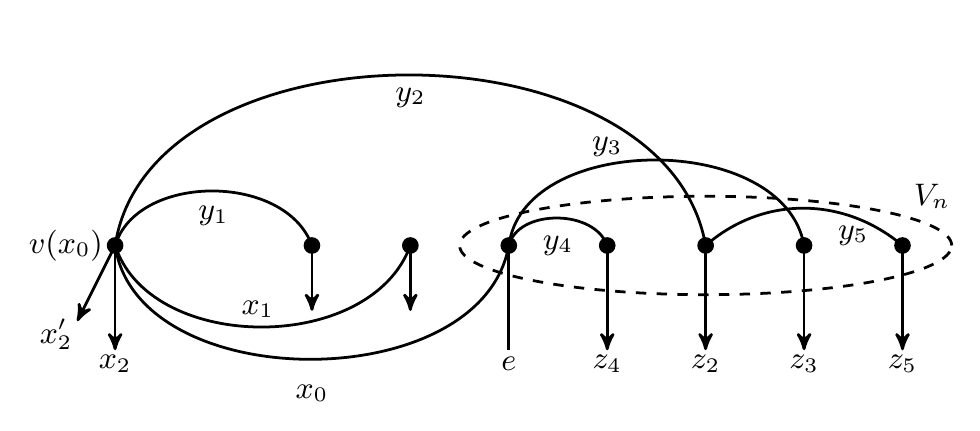}
\end{center}
\caption{A typical round of the algorithm when $m = 2$ and $x_1\notin E_n$. Free edges are denoted by arrows, and $e$ is an edge with $\lceil e/m \rceil = \lceil x_0 / m\rceil$ found in a previous round which may or may not be free. In this example, $Y_1(x_0) = \{y_1, y_2, y_3, y_4, y_5\}$ and $X_1(x_0) = \{z_2, z_3, z_4, z_5\}$. Note that each member of $Y_1(x_0) \cap E_n = \{y_2, y_3, y_4, y_5\}$ contributes exactly $(m-1)$ free edge(s) to $X_1(x_0)$. Edges $y_2, y_3, y_4, y_5, z_2, z_3, z_4, z_5$ are added to $C$, which already contains $x_0$ and $e$. Half-edges $(x_0, 1), (x_0, 2), (x_1, 1), (x_2, 2), (x_2', 2), (y_1, i), (y_2, i), (y_3, i), (y_4, i), (y_5, i)$, $i=1,2,$ are added to $R$, and edges $x_0, x_1, y_1, y_2, y_3, y_4, y_5$ to $A$. Edges $x_0, y_2, y_3, y_4, y_5$ are in $T_{x_0}$. Note that $T_{x_0}$ may contain more edges, not pictured, if some edge randomly chose $(x_0, 1)$ in a previous round.}
\label{fig:algo}
\end{figure}

{\bf Explanation of algorithm:} We call steps 1--4 a {\em round} of the algorithm. At the beginning of each round, we choose some free edge $x_0\in E_n$ that has been determined to be in $C\subseteq C_G(e_0)$. The objective of the round is to build the set $T_{x_0}$ in order to find free edges $X_1(x_0)$ which share a component with $x_0$. See Figure \ref{fig:algo} for a typical outcome of a round in which $x_1\notin E_n$. Note that part of $T_{x_0}$ may have been found in a past round.

Step 0 is a preliminary step; if $v_0\in V_n$ then we feed the $m$ free edges adjacent to $v_0$ into $X$, and if $v_0\notin V_n$ then we find $T_{v_0}$ in step 3 and feed any free edges adjacent to $T_{v_0}$ into $X$ in step 4. We call this round 0.

The edge $x_0$ makes a random choice $(x_1, j_1)$. If $x_1\in E_n$ then $T_{x_0} = T_{x_1}$ and we cut the search short and find all of $T_{x_0}$ in a future round. The reason for this is mainly to make calculations easier in Lemma \ref{Xlemma}. In the current round we will find the part of $T_{x_0}$ that can be traced back to $x_0$.

The edge $x_0$ has a fixed endpoint $\lceil x_0/m\rceil$ and a random endpoint $v(x_0)$. If $x_1\notin E_n$ then $v(x_0)\notin V_n$, and we will have $T_{x_0} = T_{v(x_0)}$. In step 2 we determine $v(x_0)$. We assign $x_0$ to $(x_1, j_1)$, and if $j_1 = 1$ we assign $x_1$ to $(x_2, j_2)$, and so on until one of two things happen. If $j_1 = j_2 = \dots = j_{k-1} = 1$ and $j_k = 2$ for some $k$ then $v(x_0) = \lceil x_k / m\rceil$. If $j_1 = \dots = j_{k} = 1$ and $x_k \leq m\n_H$, then $v(x_0) = v(x_1) = \dots = v(x_k)$, noting that $v(x_k)$ is not random when $x_k \leq m\n_H$.

At the start of any round, we have sets $A, R, X$ and a partially generated graph $\widetilde{\G}\in \mG(A, R)$ such that if $e\in A$ then $(e, 1) \in \OM(x)$ only if $x\in X$. For this reason, it is not possible that $x_j\in A$ for any $j\geq 2$, since we only consider $j\geq 2$ when $x_1\notin E_n$, so $x_1\notin X$.

Assuming $v(x_0)$ was found, in step 3 we find $T_{x_0}$ using a modification of the {\em expose} algorithm in Section \ref{distributionsec}, noting that part of the tree has already been built. We do this by exposing (i) $(e, 2)$ for the $m$ free edges $e$ adjacent to $v(x_0)$, (ii) $(e, 1)$ for all edges determined to be in $T_{x_0}$, and (iii) $(e, 2)$ for the edges in $T_{x_0}$ that are in $E_n$. We take care not to include edges in $X$, and in particular if one edge $e$ is determined to be in $X$ then we immediately place the other $m-1$ edges adjacent to $\lceil e/m\rceil$ in $X$. These rules are included to avoid $X$ decreasing in size.

Entering step 4 we have a set $Y_1(x_0)$ of non-free edges that are in $T_{x_0}$ and a set $X_1(x_0)$ of free edges whose fixed endpoint is also the fixed endpoint of some edge in $Y_1(x_0) \cap E_n$. If $x_1\in E_n$ we have $|X_1(x_0)| = m + (m-1)|Y_1(x_0)\cap E_n|$, and if $x_1\notin E_n$ then $|X_1(x_0)| = (m-1)|Y_1(x_0)\cap E_n|$. 
\\{\bf End of explanation.}

If the algorithm terminates, i.e. $X = \emptyset$ at some point, then $C = C_G(v_0)$. By estimating the round $T$ at which the algorithm terminates, we can estimate the size of $C_G(e_0)$ via Lemma \ref{BClemma} (ii) below. Let $E_c = \{e : e > mn / \om\}$ be the set of edges for which Lemma \ref{exposelemma} applies. In Lemma \ref{Xlemma} we estimate $T$ by showing that if $R\cap E_c$ (taken to mean $\{e\in E_c : (e, 1)\in R \ \mbox{or}\ (e, 2)\in R\}$) is not too large then $\{|X_t| : t\geq 0\}$ is bounded below by a random walk with positive drift.

\begin{lemma}\label{Xlemma} Suppose $m\geq 2$ and let $Z$ be a random variable taking values in $\{0, 1, 2\}$ with $\Prob{Z = 0} = 0.26$ and $\Prob{Z = 1} = 0.46$. Suppose a round starts at $x_0\in X$ and with $|R\cap E_c| \leq n^{1/2 + \e}\log_\g^3 n$. Then $|X_1(x_0)|$ is stochastically bounded below by $Z$.
\end{lemma}

The following lemma shows that if $R\cap E_c$ is too large for the bounds in Lemma \ref{Xlemma} to apply, then we have found a large component.

\begin{restatable}{lemma}{BClemma}\label{BClemma}
Let $C_t, R_t, X_t$ denote the states of $C, R, X$ after $t$ rounds of the algorithm. 
\begin{enumerate}[(i)]
\item There exists a constant $\l > 0$ such that $|R_t| \leq \l |C_t|\log_\g^3 n$ for all $t$ with probability $1 - o(n^{-1})$.
\item For all $t$, $\frac12 |C_t| \leq |X_t| + t \leq |C_t|$.
\end{enumerate}
\end{restatable}

The proofs of Lemmas \ref{Xlemma} and \ref{BClemma} are postponed to the end of this section. Suppose the algorithm is run starting at some vertex $v_0$. If at some point $|R\cap E_c| \geq n^{1/2 + \e}\log_\g^3 n$ then we conclude that $|C_G(v_0)| \geq \l^{-1} n^{1/2 + \e}$, and say that the component (and every edge and vertex in it) is {\em large}. If the algorithm terminates with $|X| = 0$ then we say that the component is {\em small}.

As long as $|R\cap E_c| < n^{1/2 + \e}\log_\g^3 n$ we will bound $|X_t|$ below by a random walk $|X_0| + \sum_{i = 1}^t (Z_i - 1)$ where the $Z_i$ are independent copies of $Z$ as defined in Lemma \ref{Xlemma}. Here $X_0$ is the state of $X$ after round $0$, and we have $|X_0| = m$ if $v_0 \in V_n$, $|X_0| = 0$ if $v_0\notin V_n$ is isolated in $G_n$, and $|X_0| \geq m-1$ if $v_0\notin V_n$ is non-isolated in $G_n$.

The rest of the proof follows from four separate claims.

{\bf Claim 1: Small components have size $O(\log n)$.} Let $X_t, R_t$ denote the states of the sets $X, R$ after $t$ rounds of the algorithm, i.e. when steps 1--4 have been executed $t$ times. Let $T$ denote the minimum $t > 0$ for which $X_t = \emptyset$. We have $|C_G(e_0)| = |C_T|$, so by Lemma \ref{BClemma} (ii), $\frac12 |C_G(e_0)| \leq T \leq |C_G(e_0)|$ with probability $1 - o(n^{-1})$. We bound the probability that $c\log n \leq T \leq n^{1/2 + \e}$ for some $c > 0$ to be chosen.

Suppose $t < T$. Since $|X_{t + 1}| \geq |X_t| - 1$ for all $t$, we must have $0 = |X_T| \geq |X_t| - (T - t)$, so $T \geq |X_t| + t$. Conditioning on Lemma \ref{BClemma}, $T \leq n^{1/2 + \e}$ implies that for all $t < T$,
$$
|R_t| \leq 2\l(|X_t| + t)\log_\g^3 n \leq 2\l n^{1/2 + \e}\log_\g^3 n,
$$
so
$$
\Prob{c\log n \leq T \leq n^{1/2 + \e}} \leq \mbox{Pr}\left\{c\log n\leq T \leq n^{1/2 + \e} \middle| |R_t| \leq 2\l n^{1/2 + \e}\log_\g^3 n \mbox{ for $t\leq T$}\right\}.
$$
Conditioning on $|R_t|\leq 2\l n^{1/2 + \e}\log_\g^3 n$, Lemma \ref{Xlemma} applies. We couple $|X_t| - |X_{t-1}|$ to independent copies $Z_t - 1$ of $Z - 1$, so if $|X_0|$ denotes the size of $X$ after round 0,
$$
|X_t| = |X_0| + \sum_{i=1}^t (|X_i| - |X_{i-1}|) \geq m-1 + \sum_{i=1}^{t} (Z_i - 1).
$$
Here $|X_0| \geq m-1$ whenever $T > 0$.

The process $W_t = m-1 + \sum_{i=1}^t (Z_i - 1)$ is a random walk with $W_{t} - W_{t-1} \in \{-1,0,1\}$ and $\E{W_{t}-W_{t-1}} = \E{Z_{t} - 1} = 0.02$. Choosing $c > 0$ large enough, Hoeffding's inequality \cite{hoef} shows that
$$
\Prob{\exists t \geq c\log n : W_t = 0} \leq \sum_{t\geq c\log n} \Prob{Z_1 + \dots + Z_t < 1.01t} = o(n^{-1}),
$$
and since $|X_t| \geq W_t$, it follows that with probability $1 - o(n^{-1})$ the algorithm either terminates after at most $c\log n$ steps, or $T \geq n^{1/2 + \e}$, in which case the component is large. Since $\frac12|C_G(e_0)| \leq T\leq |C_G(e_0)|$ with probability $1-o(n^{-1})$, and the number of components is $O(n)$, all small components have size at most $c\log n$ with high probability.

{\bf Claim 2: The probability for a non-isolated $v_0$ to be in a small component is at most $(13/14)^m$.} Recall that in a small component, $|X_t| \geq W_t$ for a random walk $W_t$ as above. Since $W_0 \geq m-1$, we have
$$
\Prob{\exists t\leq c\log n : |X_t| = 0} \leq \Prob{\exists t : W_t = 0} = \bfrac{0.26}{0.28}^{m-1} = \bfrac{13}{14}^{m-1},
$$
see e.g. \cite[Exercise 5.3.1]{grimmett}.

{\bf Claim 3: All large vertices are in the same connected component.} Suppose $v$ is a large edge and let $X_v, R_v$ be the states of $X, R$ at the point that $|R|$ hits $n^{1/2 + \e}\log_\g^3 n$ when the algorithm is run starting at $v$. Then the above shows that $|X_v| \geq c n^{1/2+\e}$ whp for some $c > 0$. Similarly, if $w$ is a large vertex then $|X_w| \geq c n^{1/2 + \e}$. Assign all edges in $X_v \cup X_w$. For every pair $e \in X_v, f \in X_w$, either $e \in E_{f}^\s$ or $f \in E_{e}^\s$, since edges in $X$ are required to be in the edge set $E_n$ of $G_n$. In particular, either half the edges $e\in X_v$ have half of $X_w$ in $E_{e}^\s$, or half the edges $f\in X_w$ have half of $X_v$ in $E_{f}^\s$. In the former case, the probability that no edge $e\in X_v$ chooses any $f\in X_w$ is bounded above by
$$
\left(1 - \frac{\Omega(n^{1/2 + \e})}{n}\right)^{\Omega(n^{1/2 + \e})} = \exp\left\{-\Omega(n^{2\e})\right\}
$$
and in the other case, the same bound holds. So with high probability, any two large edges belong to the same component. In other words, there is a unique large component.

{\bf Claim 4: The large component contains $\OM(n)$ vertices.} The number of vertices in $G_n$ is $p n + O(n^{1/2}\ln n)$ since $\s$ is $\om$-concentrated. By part (i) of Theorem \ref{giantthm}, the number of non-isolated vertices is$(1 - \xi) n + O(n^{1/2}\ln n)$ whp for some $\xi > 0$. By linearity of expectation and Claim 2, the number $S$ of small, non-isolated vertices in $G_n$ satisfies
$$
\E{S} \leq (1 - \xi)\bfrac{13}{14}^{m-1} n + O(n^{1/2}\ln n).
$$
We note that $\E{S} = \OM(n)$: when the algorithm starts with $X = \{x_1,\dots,x_m\}$, the $m$ free edges adjacent to some $v_0 \in V_n$, the probability that $X_1(x_i) = \emptyset$ for $i=1,2,\dots,m$ is bounded away from $0$.

Write $S = \sum_{v\in V} S_v$ where $S_v$ is the indicator variable for $v$ being small. Then $\E{S(S-1)} = \sum_{u\neq v} \E{S_uS_v}$. Fix $u\neq v$. Suppose we run the process starting at $v$ and find that the component is small. In the process of determining that the component is small, we assign some edges $A_v$ and expose some half-edges $R_v$, where $|A_v| = O(\log n)$ and $|R_v| = O(\log^2 n)$. The probability that $u$ is in the component is $O(\log^2 n / n)$. If $u$ is not in the component, the algorithm is run starting at $u$ on the partially generated $\widetilde{\G}\in \mG(A_v, R_v)$. In the statement of the algorithm we assumed that it is run on $\G_0 \in \mG(\emptyset, \emptyset)$, but it can be easily modified  to accommodate for $\widetilde{\G} \in \mG(A_v, R_v)$, and it will follow that $\E{S_u \mid S_v = 1} = \E{S_u}(1 + o(1))$. Hence $\E{S_uS_v} = \E{S_u}\E{S_v}(1 + o(1))$, and Chebyshev's inequality shows that $S = \E{S} + o(n)$. Since $\E{S} = \OM(n)$, this shows that with high probability,
$$
S = \E{S} + o(n) \leq (1 - \xi) \bfrac{13}{14}^{m-1}n + o(n).
$$
The theorem follows.

\subsection{Proof of Lemma \ref{Xlemma}}

We first note that $\Prob{x_1 \in A} = o(1)$, since if $x\in A$ then $(x, 1)\in R$, and $|R| = o(n)$. If $x_1\in E_n\setminus A$ then $D(x_1)\subseteq X_1(x_0)$ so $|X_1(x_0)| \geq m \geq 2$. If $x_1\notin E_n$ we have $|X_1(x_0)| = (m-1)|Y_1(x_0)\cap E_n|$. The lemma will follow from showing that for all $m\geq 2$,
$$
\Prob{|Y_1(x_0)\cap E_n| = 0 \mbox{ and } x_1\notin E_n} \leq 0.255
$$
and for $m = 2$,
$$
\Prob{|Y_1(x_0) \cap E_n| = 1 \mbox{ and } x_1\notin E_n} \leq 0.455.
$$

Throughout this proof we take $a \approx b$ to mean that $a = b + o_n(1)$. Let $(x_1, j_1)$ be the random choice of $x_0$. We first note that if $\t_0 = \log_\g(pmn / x_0) \in [0, 1]$ and $\t_1 = \log_\g(pmn / x_1)$ then for $y\in [0,1]$, 
\begin{equation}\label{x1dist}
\Prob{\t_1 - \t_0 \leq y} \approx \frac{\ln\g}{1-1/\g}\int_0^y \g^{-x}dx.
\end{equation}
Indeed, $(x_1, j_1)$ is a uniformly random member of $\widetilde{\OM}(x_0) = (E_{x_0}^\s \times [2]) \setminus R$, and since $\s$ is $\om$-concentrated we have $E_{x_0}^\s = \{x_0 / \g + O(n^{-1/2}\ln n),\dots, x_0 - i\}$ for some $i\in [m]$. Since $|R| = o(|E_{x_0}^\s|)$ and $\widetilde{\OM}(x_0) \supseteq (E_{x_0}^\s\times [2])\setminus R$, we can view $x_1$ as essentially being a uniform member of $E_{x_0}^\s$. Then $\t_1 - \t_0 = \log_\g(x_0 / x_1)$ is exponentially distributed, truncated to $[0,1]$ as in \eqref{x1dist}. In particular, since $x_1 \notin E_n$ when $\t_1 > 1+\d$ for some $\d=O(n^{-1/2}\ln n)$.
\begin{equation}\label{x1notinEn}
\Prob{x_1\notin E_n} \approx \frac{\ln \g}{1-1/\g}\int_{1-\t_0}^1 \g^{-x}dx = \frac{\g^{\t_0}-1}{\g-1} \leq \t_0.
\end{equation}

{\bf Claim A:} Let $m\geq 2$ and $x_0\in E_n$. Then $\Prob{|Y_1(x_0)\cap E_n| = 0 \mbox{ and } x_1\notin E_n} < 0.255$.

{\bf Proof of claim A:} Let $m\geq 2$ and fix an edge $x_0\in E_n$. Suppose $x_1\notin E_n$. In step 2 of the algorithm we then find a chain of edges $x_1,x_2,\dots, x_K$ for some random $K$. Since $|R| = o(n^{3/4})$ and $|\OM(x_i)| = \OM(n/\om)$ for all $x_i \geq mn/\om$, we have $\Prob{j_i = 1} = 1/2 + o(n^{-1/4})$ for all $i$, and $K$ is approximately geometric with mean $2$. In particular, since $\log_\g(x_i / x_{i-1}) \leq 1 + o(1)$ for all $i$ we have $x_K \geq mn/\om$ with probability $1 - o_n(1)$. Condition on this.

We will consider two subsets of $Y_1(x_0)$. Let $\mathcal{R}(x_0)$ be the edges found when exposing $(x_0, 1)$ and $(x_0, 2)$, and let $\mathcal{L}(x_0)$ be the set of edges in $E_n$ found by exposing $(x_1, 1), (x_2, 1),\dots, (x_{K-1}, 1)$ and $(x_K', 2)$ for all $x_K' \in D(x_K)$. Then 
$$
\Prob{|X_1(x_0) = 0|} = \Prob{|\mathcal{R}(x_0)| = |\mathcal{L}(x_0)| = 0},
$$
and we now argue that $|\mathcal{R}(x_0)|, |\mathcal{L}(x_0)|$ are essentially independent. We find $\mathcal{R}(x_0)$ by exposing $(x_0, 1)$ and $(x_0, 2)$. By Lemma \ref{exposelemma}, the number of edges found is asymptotically geometric, and in particular is $O(\log n)$ whp. Initially $|\OM(e)|$ is of order $n$ for all $e > x_0$, so exposing $O(\log n)$ edges only shrinks $\OM(e)$ to $\widetilde{\OM}(e)$ of size $|\widetilde{\OM}(e)| = |\OM(e)|(1 - o(1))$. When $|\mathcal{L}(x_0)|$ is calculated, starting with $\widetilde{\OM}(e)$ instead of $\OM(e)$ for $e > x_0$ makes an insignificant difference to the result, and we have
$$
\Prob{|\mathcal{R}(x_0)| = j_1 \ \mbox{and} \ |\mathcal{L}(x_0)| = j_2} \approx \Prob{|\mathcal{R}(x_0) = j_1}\Prob{|\mathcal{L}(x_0)| = j_2}.
$$

Let $\t_0 = \log_\g(pmn / x_0)$. Since $\s$ is $\om$-concentrated we have $\t_0\in (-\d, 1+\d)$ for some $\d = O(n^{1/2}\ln n)$, see Lemma \ref{sigmalemma}. Assume for now that $\t_0\in [0,1]$. Let $E(x_0, i)$ denote the set of edges in $E_n$ found by exposing $(x_0, i)$. By Lemma \ref{exposelemma}, $|E(x_0, i)|$ is asymptotically geometrically distributed (nonzero since $x_0 \in E_n$) with rate $e^{-\a\t_0}$ for $i=1,2$ so
\begin{equation}\label{rightprob}
\Prob{|\mathcal{R}(x_0)| = j} \approx \left\{\begin{array}{ll}
 e^{-2\a\t_0}, & j = 0, \\
2e^{-2\a\t_0}(1-e^{-\a\t_0}), & j = 1.
\end{array}\right.
\end{equation}

Now consider the chain $x_0 > x_1 > \dots > x_K$ where $x_{i-1}$ chooses $(x_i, 1)$ for $1 \leq i \leq K-1$ and $x_{K-1}$ chooses $(x_K, 2)$. If $K = 1$ and $x_1\notin E_n$, then $\Prob{|\mathcal{L}(x_0)| = 0} \approx (1-q(\t_1))^m \leq (1-q(\t_1))^2$ by Lemma \ref{exposelemma}, where $\t_1-\t_0$ can be approximated by a truncated exponential as above, so
$$
\Prob{|\mathcal{L}(x_0)| = 0, \ K = 1 \mbox{ and } x_1\notin E_n} \leq \frac12 \frac{\ln\g}{1-1/\g}\int_{1-\t_0}^1 \frac{(1-q(\t_0+x))^2}{\g^x}dx.
$$
In Claim C we show that for all $\a > 1/2$ and $\t_0\in [0,1]$,
$$
\frac12\frac{\ln\g}{1-1/\g}\int_{1-\t_0}^1 \frac{(1-q(\t_0 + x))^2}{\g^x}dx \leq \frac{\t_0}{2e - e^{1/2}}.
$$

If $K > 1$, then $\mathcal{L}(x_0) = \emptyset$ only if $E(x_1, j_1) = E(x_2, j_2) = \emptyset$. If $\t_i = \log_\g(pmn / x_i)$ denotes the age of $x_i$ then the probability of $E(x_i, j_i)$ being empty is $1-q(\t_i) \leq 1-q(\t_0 + i)$ for $i = 1,2$. Here we used the fact that $q(\t)$ is decreasing, see Lemma \ref{constlem} (\ref{decreasing}). Since $\Prob{K\geq 2} = 1/2$, we have by \eqref{x1notinEn},
$$
\Prob{|\mathcal{L}(x_0)| = 0, \ K \geq 2 \mbox{ and } x_1\notin E_n} \leq \frac{\t_0}{2} (1-q(\t_0 + 1))(1 - q(\t_0 + 2)).
$$
The function $q(\t)$ is defined in Section \ref{confun}, and we have
$$
(1-q(\t_0 + 1))(1-q(\t_0 + 2)) = \frac{1}{e^\a - \a\t_0} \frac{e^\a-\a\t_0}{e^{2\a} - (\t_0 + 1)\a e^\a + \frac12 \a^2\t_0^2}.
$$
We show in Claim C that this is at most $1 / (e - e^{1/2} + 1/8)$. So
\begin{align*}
& \Prob{|Y_1(x_0) \cap E_n| = 0 \mbox{ and } x_1\notin E_n} \\
\leq & e^{-2\a\t_0}\left(\frac12\frac{\ln\g}{1-1/\g}\int_{1-\t_0}^1 \frac{(1-q(\t_0 + x))^2}{\g^x} dx + \frac12 (1-q(\t_0+1))(1-q(\t_0+2))\right) \\
\leq & \t_0 e^{-2\a\t_0} \left(\frac{1}{2e - e^{1/2}} + \frac{1}{2(e - e^{1/2} + \frac18)}\right). 
\end{align*}
Let $L_0$ denote the expression in brackets, and note that $L_0 < 0.69$. We have $\t_0 e^{-2\a\t_0} \leq e^{-1}$ for $\a > 1/2$ and $\t_0 \in [0, 1]$, so
$$
\Prob{|Y_1(x_0)\cap E_n| = 0 \mbox{ and } x_1\notin E_n} < e^{-1}\cdot 0.69 < 0.255.
$$
{\bf End of proof of claim A.}

{\bf Claim B:} Let $m = 2$ and $x_0 \in E$. Then $\Prob{|Y_1(x_0) \cap E_n| = 1 \mbox{ and } x_1\notin E_n} < 0.455$.

{\bf Proof of claim B:} We note that while $\mathcal{L}(x_0)$ and $\mathcal{R}(x_0)$ do not necessarily partition $Y_1(x_0)\cap E_n$, it is the case that
\begin{align*}
\Prob{|Y_1(x_0)\cap E_n| = 1} & \leq \Prob{|\mathcal{L}(x_0)| = 1, |\mathcal{R}(x_0)| = 0} + \Prob{|\mathcal{L}(x_0)| = 0, |\mathcal{R}(x_0)| = 1} \\
& \approx \Prob{|\mathcal{L}(x_0)| = 1}\Prob{|\mathcal{R}(x_0)| = 0} + \Prob{|\mathcal{L}(x_0)| = 0}\Prob{|\mathcal{R}(x_0)| = 1}.
\end{align*}
We calculated the probability that $|\mathcal{R}(x_0)| = 0, 1$ in \eqref{rightprob}. For the probability that $|\mathcal{L}(x_0)| = 1$, let $x_{K + 1}$ denote the edge added along with $x_K$ (so $|x_{K+1} - x_K| = 1$). Let $\t_i = \log_\g (pmn / x_i)$ for $i = 0,1,\dots, K +1$. Then
\begin{align*}
\Prob{|\mathcal{L}(x_0)| = 1 \mid x_1\notin E_n}
\leq & \sum_{k\geq 1}\Prob{K = k}\sum_{i = 1}^{k + 1} q(\t_i)p(\t_i)\prod_{\substack{1\leq j \leq k+1 \\ j\neq i}} (1-q(\t_j))
\end{align*}
where $i$ denotes the edge whose exposure contributes to $\mathcal{L}(x_0)$. We use the bound $1-q(\t_1) \leq 1-q(\t_0 + 1)$ whenever $1-q(\t_1)$ is involved in the product (i.e. when $i > 1$), and bound $p(\t_i)q(\t_i) \leq p(1)q(1)$ for all $i\geq 1$ (which follows from $p(\t), q(\t)$ being decreasing, see Lemma \ref{constlem} (\ref{decreasing})) to get
\begin{align*}
\Prob{|\mathcal{L}(x_0)| = 1\mid x_1\notin E_n}
& \leq  \sum_{k \geq 1} \frac{1}{2^k} p(1)q(1)\left(1 + k(1 - q(\t_0+1))\right) \\
& = e^{-\a}(1-e^{-\a})\left(1 + \frac{2}{e^\a-\a\t_0}\right) \\
& \leq \frac14\left(1 + \frac{2}{e^{1/2}-1/2}\right).
\end{align*}
This bound holds for all $\a > 1/2, \t_0\in[0,1]$. Let $L_1 = 1/4 + 1/(2e^{1/2}-1)$.

We now bound
\begin{align*}
\Prob{|Y_1(x_0) \cap E_n| = 1 \mbox{ and } x_1\notin E_n} \leq & \Prob{|\mathcal{R}(x_0)| = 0}\Prob{|\mathcal{L}(x_0)| = 1\mid x_1\notin E_n}\Prob{x_1\notin E_n} \\
& + \Prob{|\mathcal{R}(x_0)| = 1} \Prob{|\mathcal{L}(x_0)| = 0 \mbox{ and } x_1\notin E_n} \\
\leq & \t_0 e^{-2\a\t_0}L_1 + 2\t_0 e^{-2\a\t_0}(1-e^{-\a\t_0})L_0 \\
\leq & \frac1e L_1 + \frac{1}{\a e}(1-e^{-\a})L_0,
\end{align*}
where we used the fact that $\t_0 e^{-2\a\t_0}$ viewed as a function of $\t_0$ has a global maximum at $\t_0 = 1/2\a$, so $\t_0e^{-2\a\t_0}\leq 1/(2\a e) \leq 1/e$, and we also used $1 - e^{-\a\t_0} \leq 1-e^{-\a}$. Finally, $(1-e^{-\a})/(\a e)$ is decreasing in $\a$, so
$$
\Prob{|Y_1(x_0)\cap E_n| = 1 \mbox{ and } x_1\notin E_n} \leq \frac{L_1}{e} + \frac{2}{e}(1-e^{-1/2})L_0  < 0.455
$$
{\bf End of proof of claim B.}

{\bf Claim C:} The following two inequalities hold for all $\a > 1/2$ and $\t_0\in [0,1]$:
\begin{equation}\label{C1}
(1-q(\t_0 + 1))(1-q(\t_0 + 2)) \leq \frac{1}{e - e^{1/2} + 1 / 8},
\end{equation}
and
\begin{equation}\label{C2}
\frac{\ln\g}{2-2/\g}\int_{1-\t_0}^1 \frac{(1-q(\t_0 + x))^2}{\g^x}dx \leq \frac{\t_0}{2e - e^{1/2}}.
\end{equation}
{\bf Proof of claim C:} To emphasize the dependence on $\a$ we briefly write $q(\a, \t) = q(\t)$. For $\t_0 \in [0, 1]$ we have
$$
q(\a, \t_0 + 1) = \frac{1}{e^\a - \a\t_0},\quad q(\a, \t_0 + 2) = \frac{e^\a - \a\t_0}{e^{2\a} - (\t_0+1)\a e^\a + \frac12 \a^2 \t_0^2}.
$$
Suppose $\a_1 > \a_2$ and let $\mC_1$ be a CMJ process with rate $\a_1$. Mark any arrival red with probability $\a_2 / \a_1$, and consider the CMJ process $\mC_r$ on the red arrivals. This will have rate $\a_2$, and if $\mC_r$ is active at time $\t_0$ then so is $\mC$. This implies $q(\a_1, \t) \geq q(\a_2, \t)$ for all $\t$, since $q(\a, \t)$ is the probability that a CMJ process of rate $\a$ is active at time $\t$. So for any $\a > 1/2, \t_0 \in [0, 1]$, 
\begin{equation}\label{C11}
1 - q(\a, \t_0 + 1) \leq 1 - q\left(\frac12, \t_0 + 1\right) = \frac{1}{e^{1/2} - \t_0 / 2}
\end{equation}
and
\begin{equation}\label{C12}
1-q(\a, \t_0 + 2) \leq 1-q\left(\frac12, \t_0 + 2\right) = \frac{e^{1/2} - \t_0/2}{e - \frac{\t_0+1}{2}e^{1/2} + \t_0^2 / 8}.
\end{equation}
Consider multiplying \eqref{C11} and \eqref{C12}. It is easy to confirm that $e - \frac{\t_0+1}{2}e^{1/2} + \t_0^2 / 8$ is decreasing for $\t_0\in [0,1]$, and \eqref{C1} follows.

Now consider \eqref{C2}. First note that $\a = \frac{p}{4p-2}\ln \g = \frac{1}{2-2/\g}\ln \g$. We have
$$
\frac{\ln\g}{2-2/\g}\int_{1-\t_0}^1 \frac{(1-q(\t_0+x))^2}{\g^x} dx = \int_{1-\t_0}^1 \frac{\a}{\g^x (e^\a-\a (x + \t_0 - 1))^2}dx = \int_0^{\t_0} \frac{\a}{\g^{x+1 - \t_0}(e^\a-\a x)^2} dx.
$$
Fix $\t_0$ and let $f(\a, x) = \a/(\g^{x+1-\t_0}(e^\a - \a x)^2)$ for $0 < x < \t_0$. We will show that $f(\a, x) \leq \lim_{\a\to 1/2} f(\a, x)$ for $\a > 1/2$ by showing that $f(\a, x)$ is decreasing in $\a$. To calculate the derivative of $\g^{-(x+1-\t_0)}$ with respect to $\a$, we note that since $\a = \frac{1}{2-2/\g}\ln \g$,
$$
\frac{d\g}{d\a} = \frac{(2\g-2)^2}{2\g-2 - 2\ln \g} = \frac{2\g-2}{1 - \frac{1}{\g-1}\ln \g} = \frac{2\g-2}{1 - 2\a/\g}.
$$
Since $\ln \g < \g - 1$ we have $1 < 2\a = \ln\g/(1-1/\g) < \g$, so
$$
\frac{d\g}{d\a} = 2\g \frac{\g-1}{\g-2\a} > 2\g.
$$
In particular,
$$
\frac{d}{d\a} \g^{-(x + 1 - \t_0)} = -(x + 1 - \t_0) \g^{-(x + 1 - \t_0)}\frac1\g \frac{d\g}{d\a} < -2(x+1-\t_0)\g^{-(x+1-\t_0)}.
$$

Now for $0\leq x\leq \t_0\leq 1$ and $\a > 1/2$, since $e^\a > 1/2 + \a x$ we have
\begin{align*}
\frac{\partial f}{\partial\a} & = \frac{1}{\g^{x+1-\t_0}(e^\a-\a x)^2} + \frac{\a}{(e^\a - \a x)^2} \left(\frac{d}{d\a} \g^{-(x-\t_0+1)}\right) - \frac{2\a(e^\a - x)}{\g^{x-\t_0 + 1}(e^\a - \a x)^3} \\
&< \frac{1}{\g^{x+1-\t_0}(e^\a - \a x)^2} - \frac{2(x+1-\t_0) \a}{\g^{x+1-\t_0}(e^\a - \a x)^2} - \frac{2\a (e^\a - x)}{\g^{x-\t_0+1}(e^\a - \a x)^3} \\
& = \frac{1}{\g^{x + 1-\t_0}(e^\a - \a x)^3}\left(e^\a - \a x - 2(x + 1 - \t_0)\a (e^\a - \a x) - 2\a(e^\a - x)\right) \\
& < \frac{1}{\g^{x + 1 - \t_0}(e^\a - \a x)^3}\left(e^\a - \a x - 2x\a(e^\a - \a x) - 2\a(e^\a - x)\right) \\
& = \frac{1}{\g^{x+1-\t_0}(e^\a - \a x)^3}\left(e^\a(1-2\a) -2\a x(e^\a - \a x - 1/2)\right) \\
& < 0.
\end{align*}
Noting that $\g \to 1$ as $\a\to 1/2$, this implies
$$
\int_0^{\t_0} \frac{\a}{\g^{x+1-\t_0}(e^\a - \a x)^2}dx < \int_0^{\t_0} \frac{1/2}{(e^{1/2} - x/2)^2} = \frac{1}{e^{1/2} - \t_0/2} - \frac{1}{e^{1/2}} = \frac{\t_0}{2e^{1/2}(e^{1/2} - \t_0/2)}.
$$
Then \eqref{C2} follows from $e^{1/2} - \t_0 / 2 \geq e^{1/2} - 1/2$.
\\{\bf End of proof of claim C.}

\subsection{Proof of Lemma \ref{BClemma}}

Recall Lemma \ref{BClemma}:

\BClemma*

{\bf Proof of (i).} The key observation is that by Lemma \ref{CMJbirthlem} (iii) and Lemma \ref{exposelemma}, if $e > mn/\om$ and we expose $(e, j)$ then there exists a $\l > 0$ such that $|E(e, j) \cap E_n| \geq \lfloor |E(e, j)| / (\l \log_\g^2 n)\rfloor$ with probability $1 - o(n^{-1})$ . Here $E(e, j)$ denotes the set of edges found when exposing $(e, j)$. Condition on this being the case for all $O(n)$ half-edges exposed over the course of the algorithm. To avoid rounding, we note that if $|E(e, j) \cap E_n| = 0$ then $|E(e, j)| \leq \l\log_\g^2 n$ and if $|E(e, j)\cap E_n| > 0$ then $|E(e, j)| \leq 2\l |E(e,j)\cap E_n| \log_\g^2 n$.

The above holds if $e > mn/\om$. If $e \leq mn/\om$ and $(e, j)\in Q(x)$, Lemma \ref{exposelemma} does not apply to exposing $(e, j)$. In this case, reveal $(e, j)$, i.e. find all $f$ such that $\f(f) = (e, j)$. Note that
$$
E(e, j) = \{(e, j)\} \cup \bigcup_{(f, 1) : f\in \f^{-1}(e, j)} E(f, 1).
$$
Remove $(e, j)$ from $Q(x)$ and replace it by $(f, 1)$ for all $f\in \f^{-1}(e, j)$. Repeat this until all $(e, j) \in Q(x)$ have $e > mn/\om$. Let $Q'(x)$ be the end result of this process.

Recall that $E_c$ is the set of edges $e$ with $e > mn/\om$. We have
$$
|R_t\cap E_c| \leq \sum_{i = 1}^t 2|Y_1(x_i)\cap E_c|, \quad |C_t| \geq \sum_{i = 1}^t |Y_1(x_i)\cap E_n|
$$
and in round $i$,
$$
|Y_1(x_i)\cap E_c| = \sum_{(e, j) \in Q'(x_i)} |E(e, j)|, \quad |Y_1(x_i)\cap E_n| = \sum_{(e, j)\in Q'(x_i)} |E(e, j)\cap E_n|.
$$
Letting $(e_1, j_1),\dots,(e_s, j_s) \in \cup_i Q'(x_i)$ be the half-edges exposed in the first $t$ rounds of the algorithm, we then have
$$
|R_t\cap E_c| \leq 2\sum_{i=1}^s |E(e_i, j_i)|, \quad |C_t| \geq \sum_{i=1}^s |E(e_i, j_i)\cap E_n|.
$$
Let $I_i$ be the indicator variable for $|E(e_i, j_i) \cap E_n| > 0$, and let $I = I_1 + \dots + I_s$. Then by the above,
$$
|R_t\cap E_c| \leq 2(s-I)\l \log_\g^2 n + 2\l\log_\g^2 n\sum_{i : I_i = 1} |E(e_i, j_i)\cap E_n| = 2(s-I)\l \log_\g^2 n + 2\l|C_t|\log_\g^2n,
$$
and we will show that $s \leq I\log n \leq |C_t|\log n$.

Every edge exposed in the process is in $E_c$, so the probability that $I_i = 1$ is, by Lemma \ref{exposelemma}, $q(\t_i)$ where $\t_i = \log_\g(pmn / e_i)$. For all $i$, $\t_i \leq \log_\g \om$, and $q(\t)$ is decreasing by Lemma \ref{constlem} (\ref{decreasing}), so $I_i = 1$ with probability at least $q(\log_\g \om) \geq \l_2 \zeta^{-\log_\g \om}$ where $\l_2 > 0$, see Lemma \ref{constlem}. Let $c > 0$ be such that $q(\t_i) \geq \om^{-c}$ for all $i$. Then $I$ can be bounded below by a binomial random variable $J \sim \mbox{Bin}(s, \om^{-c})$. Suppose $s > 4\om^{2c}\log n$. Then Hoeffding's inequality \cite{hoef} implies
$$
\Prob{I < s\om^{-c}} \leq \Prob{J < s\om^{-c}} \leq\exp\left\{-2\bfrac{\om^{-c}}{2}^2 s\right\} \leq n^{-2}.
$$
Since $|C_t| \geq I$, This shows that with high probability, if $s > 4\om^{2c}\log n$ then $s\leq I\om^c \leq |C_t|\om^c$ and
$$
|R_t\cap E_c| \leq 2(s-I)\l\log_\g^2 n + 2\l|C_t|\log_\g^2 n \leq 3\l |C_t|\om^c \log_\g^2 n.
$$
If $s \leq 4\om^{2c}\log n$ then $|C_t| \geq 0$ implies
$$
|R_t\cap E_c| \leq 4\l\om^{2c}\log_\g^3 n \leq 4\l\om^{2c}(|C_t| + 1)\log_\g^3 n,
$$
and since $\om^{2c} =(\log\log n)^{2c}\leq \log_\g n$ for $n$ large enough, this finishes the proof of (i).

{\bf Proof of (ii).} In each round we have $|X_1(x)| = m + (m-1)|Y_1(x)\cap E_n|$ if $x_1\in E_n$ and $|X_1(x)| = (m-1)|Y_1(x)\cap E_n|$ if $x_1\notin E_n$. In particular, $|Y_1(x)\cap E_n| \leq |X_1(x)| / (m-1) \leq |X_1(x)|$. If $x_i$ denotes the starting edge of round $i$ then
$$
|C_t| = m + \sum_{i = 1}^t |X_1(x_i)| + |Y_1(x_i)\cap E_n|, \quad |X_t| = m + \sum_{i=1}^t |X_1(x_i)| - 1,
$$
so
$$
|C_t| - |X_t| - t = \sum_{i=1}^t |Y_1(x_i)\cap E_n| \leq \sum_{i=1}^t |X_1(x_i)| = |X_t| + t.
$$
It follows immediately that $\frac12|C_t| \leq |X_t| + t \leq |C_t|$.

\section{Concluding remarks}

The main computational task in improving the results of this paper is in estimating integral involving $p(\t), q(\t)$ and $\g^{-\t}$. To find the exact number of vertices of degree $k$ for $k = O(1)$, one needs to calculate integrals involving terms of the form $\g^{-\t}q(\t)p(\t)(1-p(\t))^{k-1}$, and this is difficult to do in any generality. Integrals involving $p(\t), q(\t)$ and $\g^{-\t}$ also appear when looking for small components, which prevented us from finding the exact size of the giant component.

\bibliographystyle{plain}
\bibliography{biblio}

\appendix

\section{Proof of Lemma \ref{constlem}}\label{constlemproofsec}

In this section we prove Lemma \ref{constlem}, in which we collect useful properties of the central constants and functions defined in Section \ref{confun}. We will restate the definitions here to make the appendix self-contained. Firstly, the integer $m\geq 1$ and the real number $1/2 < p < 1$ are the parameters for the graph process, and we define
$$
\m = m(2p - 1), \quad \g = \frac{p}{1 - p}, \quad \a = \frac{pm}{2\m}\ln \g = \frac{p}{4p-2}\ln\g.
$$
We let $p_0 \approx 0.83113$ be the unique $p$ for which $\a = 1$, and when $\a \neq 1$ we define $\zeta$ as the unique solution in $\mathbb{R}\setminus\{1\}$ to
\begin{equation}\label{zetadef}
\zeta e^{\a(1-\zeta)} = 1.
\end{equation}
If $\a > 1$ define $\eta = -\ln\g / \ln\zeta$. If $\a < 1$ then $\eta$ is undefined.

We define a sequence $a_k$ by $a_0 = 1$ and
\begin{equation}\label{akdef}
a_k = \left(-\frac{e^\a}{\a}\right) \sum_{j=0}^{k-1} \frac{a_j}{(k-j-1)!}, \ k \geq 1.
\end{equation}
For $k \geq 0$ define functions $Q_k : [k, k + 1) \to [0, 1]$ by
$$
Q_k(\t) = \sum_{j=0}^k \frac{a_j}{(k-j)!} (\t-k)^{k-j},
$$
and for $\t \geq 0$ we let $Q(\t) = Q_{\lfloor \t\rfloor}(\t)$. We note that $Q(\t)$ is discontinuous at integer points $k$ with
\begin{equation}\label{jumpsize}
Q(k) = a_k \quad \mbox{and} \quad Q(k^-) = -\a e^{-\a} a_k
\end{equation}
where $Q(k^-)$ denotes the limit of $Q(\t)$ as $\t\to k$ from below. Define
$$
q(\t) = 1, 0\leq \t < 1, \ q(\t) = 1 + \frac{Q(\t - 1)}{\a Q(\t)}, \ \t \geq 1.
$$
Finally, define
$$
p(\t) = \exp\left\{-\a\int_0^\t q(x)dx\right\}.
$$
For $\t < 0$ we define $Q(\t) = q(\t) = p(\t) = 0$.

\constlem*

\begin{proof}
{\bf Proof of (\ref{zetalemma}).} Let $\a \neq 1$. The function $x \mapsto x e^{\a(1-x)}$ is strictly increasing for $x < \a^{-1}$ and strictly decreasing for $x >\a^{-1}$, and its global maximum at $x = \a^{-1}$ is $\a^{-1} e^{\a-1} > 1$. The two solutions $x_1, x_2$ of $x e^{\a(1-x)} = 1$ must satisfy $x_1 < \a^{-1} < x_2$, and $\zeta < \a^{-1}$ for $\a > 1$ follows from the fact that $\zeta$ is the solution which is not $1$. When $\a < 1$, it is straightforward to plug in $x = 1 - \a^{-1} + \a^{-2}$ and confirm that $xe^{\a(1-x)} > 1$, which shows that $\zeta > 1 - \a^{-1} + \a^{-2} > \a^{-1}$.

{\bf Proof of (\ref{etalemma}).} Let $\a > 1$, so $p > p_0 \approx 0.83$. To see that $\eta > 2$, we first note that the definition of $\a$ gives$\ln \g = \a(4 - 2/p)$ and the definition of $\zeta$ gives $\ln \zeta = -\a(1-\zeta)$, so
$$
\eta = -\frac{\ln\g}{\ln\zeta} = \frac{4-\frac2p}{1-\zeta} > 1
$$
since $4-\frac2p > 1$ for $p > p_0\approx 0.83$ and $1-\zeta < 1-\a^{-1} < 1$ by (\ref{zetalemma}). Now, $(4-2/p)/(1-\zeta) > 2$ is equivalent to $\zeta + 1 - 1/p > 0$, and $\eta > 1$ and $\g > 1$ implies
$$
\zeta + 1 - \frac1p = \g^{\ln \zeta / \ln \g} - \frac{1-p}{p} = \g^{-1/\eta} - \g^{-1} > 0.
$$

{\bf Proof of (\ref{decreasing}).} Lemma \ref{CMJlemma} shows that $q(\t) = \Prob{X > \t}$ for a random variable $X$, namely $X = \min\{x > 0 : d(x) = 0\}$ in the notation of Lemma \ref{CMJlemma}, and (\ref{decreasing}) follows immediately.

{\bf Proof of (\ref{Qeat}).} Suppose $k\geq 1$ is an integer such that $k < \t < k + 1$. Then (\ref{Qeat}) follows from the fact that
$$
Q'(\t) = \frac{d}{d\t} \sum_{j = 0}^k \frac{a_j}{(k-j)!} (\t-k)^{k-j} = \sum_{j=0}^{k - 1} \frac{a_j}{(k-j-1)!} (\t - k)^{k-j-1} = Q(\t-1).
$$
The case $\t < 1$ follows from the fact that $Q(x) = 0$ for all $x < 0$.

{\bf Proof of (\ref{smallalpha}), (\ref{largealpha}).} We now need to look closer at the sequence $\{a_k\}$. Let $A(z)$ denote its generating function. From \eqref{akdef} we have
\begin{align*}
A(z) = 1 + \sum_{k\geq 0} z^k \left(-\frac{e^\a}{\a}\right)\sum_{j=0}^{k-1} \frac{a_j}{(k-j-1)!} & = 1  - \frac{e^\a}{\a}\sum_{j=0}^\infty a_jz^{j+1} \sum_{k=j+1}^\infty \frac{z^{k-j-1}}{(k-j-1)!}\\
& = 1-\frac{e^\a}{\a} z e^z A(z)
\end{align*}
so $A(z) = 1 / (1 + \a^{-1} z e^{\a+z})$. The sequence $b_k = (-\a)^k a_k$ then has generating function
$$
B(z) = A(-\a z) = \frac{1}{1 - z e^{\a(1-z)}}.
$$
This has simple poles at $z = 1$ and $z = \zeta$, with residues $1/(\a - 1)$ and $\zeta / (\a\zeta - 1)$ respectively. Then
$$
\overline{B}(z) = \frac{1}{1-ze^{\a(1-z)}} - \frac{1}{(1-\a)(1-z)} - \frac{\zeta}{(1 - \a\zeta)(\zeta-z)}
$$
is analytic. Writing $\b = (1 - \a\zeta)^{-1}$, the power series representation of $\overline{B}(z)$ is $b_k - 1/(1-\a) - \b/\zeta^k$. Since $\overline{B}(z)$ is analytic, Cauchy's integral formula shows that for any $\e > 0$,
$$
b_k = \frac{1}{1 - \a} + \frac{\b}{\zeta^k} + O_k(\e^k).
$$
In the remainder of the proof, fix $0 < \e < \zeta^{-1}$.

Using (\ref{Qeat}) and \eqref{jumpsize} we have, for any integer $k\geq 0$,
$$
p(k) = \exp\left\{-\a\int_0^k q(x)dx\right\} = \prod_{j = 1}^k \frac{Q(j-1)e^{\a (j - 1)}}{Q(j^-) e^{\a j}} = \frac{1}{(-\a)^k a_k} = \frac{1}{b_k}.
$$
and $q(k) = 1+Q(k - 1) / (\a Q(k)) = 1 - b_{k-1} / b_k$. If $\a < 1$ then $\zeta > 1$ so for integers $k$,
$$
q(k)\zeta^k = \zeta^k\left(1-\frac{b_{k-1}}{b_k}\right) = \frac{\zeta^k b_k - \zeta^k b_{k-1}}{b_k} = \frac{\b - \zeta\b + O(\zeta^{-k})}{\frac{1}{1-\a} + O(\zeta^{-k})} = \b(1-\a)(1-\zeta) + O(\zeta^{-k})
$$
and we set $\l_2 = \b(1-\a)(1-\zeta) = (1-\a)(\zeta-1)/(\a\zeta-1)$. Recall that $p(k) = 1/b_k$. By Taylor's theorem we have $1 / (a + bx) = a^{-1} - ba^{-2}x + O(x^2)$ for any constants $a,b\neq 0$, so with $x = \zeta^{-k}$
$$
(p(k) - (1-\a))\zeta^k = \left(\frac{1}{\frac{1}{1-\a} + \frac{\b}{\zeta^k}+O(\e^k)} - (1-\a)\right)\zeta^k = -\b(1-\a)^2 + O(\zeta^{-k})
$$
and we set $\l_1 = -\b(1-\a)^2 = (1-\a)^2 / (\a\zeta-1)$. Here $\zeta > 1-\a^{-1}+\a^{-2}$ (from (\ref{zetalemma})) implies $\l_1 < \a$.

Suppose $\a > 1$. Then $0 < \zeta < \a^{-1}$ and
$$
\frac{q(k) - (1 - \zeta)}{\zeta^k} = \frac{\zeta b_k - b_{k-1}}{\zeta^k b_k} = \frac{\frac{\zeta - 1}{1 - \a} + O(\e^{k-1})}{\b + O(\zeta^k)} = \frac{(1-\zeta)(1-\a\zeta)}{\a-1} + O(\zeta^k)
$$
and we set $\l_4 = (1-\zeta)(1-\a\zeta)/(\a-1)$. From the definition \eqref{zetadef} of $\zeta$ we have
$$
\frac{p(\t)}{\zeta^\t} = \exp\left\{-\a\int_0^\t q(x) dx - \t\ln \zeta\right\} = \exp\left\{-\a\int_0^\t (q(x) - (1-\zeta)) dx\right\}
$$
and since $q(\t)$ decreases toward $1 - \zeta$ at an exponential rate, the integral converges as $\t\to \infty$ and $p(\t)\zeta^{-\t}$ is decreasing. Again considering integer values $k$, we have
$$
\frac{p(k)}{\zeta^k} = \frac{1}{b_k\zeta^k} = \frac{1}{\b + O(\zeta^k)} = 1-\a\zeta + O(\zeta^k)
$$
and we set $\l_3 = 1-\a\zeta$.

The above shows the asymptotic behaviour of $p(\t), q(\t)$ for integer values of $\t$. Since both functions are monotone, the same asymptotics apply to non-integer values of $\t$. From (\ref{zetalemma}) it follows that $\l_1, \l_2, \l_3, \l_4$ all are positive.
\end{proof}

\section{Proof of Lemmas \ref{CMJlemma}, \ref{CMJbirthlem}}\label{CMJproofsec}

Recall Lemma \ref{CMJlemma}.

\CMJlemma*

\begin{proof}[Proof of Lemma \ref{CMJlemma}]
The process is a Crump-Mode-Jagers process, a class of processes which were studied in general in companion papers \cite{cm1}, \cite{cm2}. Define
$$
F(s, \t) = \sum_{k\geq 0} \Prob{d(\t) = k} s^k.
$$
In \cite{cm2} it is shown that the probability generating function satisfies
\begin{align}
F(s, \t) & = s \exp\left\{\a \int_0^\t (F(s, u) - 1) \ du \right\}, 0 \leq \t < 1 \label{cheq1} \\
F(s, \t) & = \exp\left\{\a\int_{\t-1}^\t (F(s, u) - 1) \ du\right\}, \t > 1. \label{cheq2}
\end{align}
We will show that $F(s,\t) = \widetilde{F}(s, \t)$ where
$$
\widetilde{F}(s, \t) = 1 - q(\t) + \frac{p(\t)q(\t)s}{1 - s(1 - p(\t))} = 1 + \frac{q(\t)(s - 1)}{1 - s(1 - p(\t))}
$$
with $p(\t), q(\t)$ defined in Section \ref{confun}. This is the probability generating function of $G(p(\t), q(\t))$.

Firstly, when $0 \leq \t < 1$ we plug $q(\t) = 1$, $p(\t) = e^{-\a\t}$ into \eqref{cheq1}, and via the integral substitution $w = e^{\a u}$,
\begin{align*}
s\exp\left\{\a \int_0^\t \left(1 + \frac{s-1}{1 - s(1 - e^{-\a u})} - 1\right) \ du\right\}
& = s\exp\left\{\a\int_0^\t \frac{(s - 1)e^{\a u}}{s + (1-s)e^{\a u}} du\right\} \\
& = s\exp\left\{\int_1^{e^{\a\t}} \frac{s-1}{s + w(1-s)} dw\right\} \\
& = s\exp\left\{- \ln(s - (s-1)e^{\a \t})\right\} \\
& = \frac{s e^{-\a \t}}{1 - s(1 - e^{-\a \t})}
\end{align*}
confirming that $\widetilde{F}(s,\t)$ satisfies \eqref{cheq1}.

For $\t > 1$ we have
$$
\exp\left\{\a\int_{\t-1}^\t (\widetilde{F}(s, u) - 1)\ du\right\} = \exp\left\{\a\int_{\t-1}^\t \frac{q(u)(s-1)}{1-s+sp(u)}\right\} du
$$
and since $p(u) = \exp\left\{-\a\int_0^u q(x)dx\right\}$, the substitution $v(u) = \ln p(u)$ with $dv/du = -\a q(u)$ yields
\begin{align*}
\a\int_{\t-1}^\t \frac{q(u)(s-1)}{1-s+sp(u)}du  & = \int_{v(\t-1)}^{v(\t)} \frac{1-s}{1-s + se^v}dv \\
& = \int_{v(\t-1)}^{v(\t)} \left(1 - \frac{s e^v}{1-s+se^v}\right)dv
\end{align*}
and substituting $w = e^v$ gives, as above,
\begin{align*}
\int_{v(\t-1)}^{v(\t)} \left(1 - \frac{s e^v}{1-s+se^v}\right) dv
& = v(\t) - v(\t-1) + \int_{e^{v(\t-1)}}^{e^{v(\t)}} \frac{s}{1-s+sw} dw \\
& = v(\t) - v(\t-1) +\ln\left(\frac{1-s+s e^{v(\t-1)}}{1-s+s e^{v(\t)}}\right).
\end{align*}
So since $v(u) = \ln p(u)$,
\begin{equation}\label{step1}
\exp\left\{\a\int_{\t-1}^\t (\widetilde{F}(s, u) - 1) \ du\right\} = \frac{p(\t)}{p(\t-1)} \frac{1 - s + s p(\t-1)}{1 - s + s p(\t)}.
\end{equation}
We have $1 - q(\t) = p(\t) / p(\t - 1)$ (see \eqref{Qid}), so
\begin{equation}\label{step2}
\frac{p(\t)}{p(\t-1)}\frac{1 - s + s p(\t-1)}{1 - s + sp(\t)} = \frac{(1-s)(1-q(\t)) + sp(\t)}{1 - s + sp(\t)} = 1 + \frac{q(\t)(s-1)}{1-s + sp(\t)} = \widetilde{F}(s, \t).
\end{equation}
Now \eqref{step1} and \eqref{step2} imply that $\widetilde{F}(s, \t)$ satisfies \eqref{cheq2}.

To see that $1 - q(\t) = p(\t) / p(\t-1)$, recall from \eqref{jumpsize} that $Q(k) / Q(k^-) = -1/(\a e^\a)$ for integers $k$, and from Lemma \ref{constlem} (\ref{Qeat}) we have $q(\t) = \a^{-1} (Q(\t)e^{\a\t})' / (Q(\t)e^{\a\t})$ for non-integer values of $\t$. So the integral of $q(\t)$ is $\a^{-1}\ln(Q(\t) e^{\a\t})$, and
\begin{align}
\frac{p(\t)}{p(\t-1)} 
& = \exp\left\{-\a\int_{\t-1}^\t q(x)dx\right\}  \nonumber \\
& = \exp\left\{-\a\int_{\t-1}^{\lfloor \t\rfloor} q(x)dx\right\}\exp\left\{-\a\int_{\lfloor \t\rfloor}^\t q(x)dx\right\} \nonumber \\ 
& =  \frac{Q(\t-1)e^{\a(\t-1)}}{Q(\lfloor \t\rfloor^-) e^{\a\lfloor \t\rfloor}}\frac{Q(\lfloor \t\rfloor) e^{\a\lfloor \t\rfloor}}{Q(\t)e^{\a\t}} \nonumber\\
& = \frac{-Q(\t-1)}{\a Q(\t)} \nonumber \\
& = 1-q(\t). \label{Qid}
\end{align}
The last equality comes form the definition of $q(\t)$.
\end{proof}

Recall Lemma \ref{CMJbirthlem}.

\CMJbirthlem*

\begin{proof}[Proof of Lemma \ref{CMJbirthlem}]
Each Poisson process has lifetime exactly $1$, so
$$
\sum_{k=0}^{\lfloor\t\rfloor} d(k) \leq b(\t) \leq \sum_{k=0}^{\lceil \t\rceil}d(k)
$$
and in particular,
\begin{equation}\label{pigeonhole}
b(\t) \leq \lceil \t\rceil \max_{0\leq k\leq \lceil\t\rceil} d(k).
\end{equation}
From Lemma \ref{CMJlemma} we have
$$
\Prob{d(\t) > \ell} = q(\t)(1-p(\t))^\ell.
$$
For $\a < 1$, Lemma \ref{constlem} (\ref{zetalemma}), (\ref{smallalpha}) imply that $1-p(\t) \leq \a$, so
$$
\Prob{\max_{0\leq k\leq \lceil\t\rceil} d(k) > -2 \log_\a n} \leq \lceil \t\rceil \a^{-2\log_\a n} = o(n^{-1}).
$$
For $\a > 1$, Lemma \ref{constlem} (\ref{zetalemma}), (\ref{largealpha}) imply
$$
\Prob{\max_{0\leq k\leq\lceil\t\rceil} d(k) > \l n^{1/\eta}\ln n} \leq \lceil\t\rceil  (1 - \l_3 \zeta^{\lceil\t\rceil})^{\l n^{1/\eta}\ln n} \leq \lceil\t\rceil  \exp\left\{-\l \l_3 \zeta^{\lceil\t\rceil}n^{1/\eta}\ln n\right\}
$$
and since $\t\leq\log_\g n$ and $\zeta^{\log_\g n}n^{1/\eta} = 1$, this is $o(n^{-1})$ for $\l$ large enough.

Assertion (iii) follows from (i) for $\a < 1$. Suppose $\a > 1$. The claim will follow from showing that we can choose $A, B > 0$ so that if $\t\leq \log_\g n$,
\begin{equation}\label{CMJsurvival}
\Prob{\exists x\in [0,\t] : d(x) \geq A\log_\g n \mbox{ and } d(\t) \leq d(x) / B} = o(n^{-1}).
\end{equation}
Indeed, suppose $b(\t) \geq A(\log_\g n)^2$. Then by \eqref{pigeonhole} there exists some $x < \t$ for which $d(x) \geq b(\t)/\t \geq A\log_\g n$. It will follow from \eqref{CMJsurvival} that $d(\t) \geq b(\t) / (B\t) \geq AB^{-1}b(\t) / \log_\g n$ with probability $1 - o(n^{-1})$. If $b(\t) < A(\log_\g n)^2$ we choose $\l > A$ so that $d(\t) \geq 0 = \lfloor b(\t) / (C\log_\g^2 n)\rfloor$.

If $x' < \t$ is such that $d(x') \geq A\log_\g n$ Poisson processes are active, then either (i) at least $d(x') / 2$ of the processes are still active at time $x' + 1/2$, or (ii) at least $d(x') / 2$ of the processes were active at time $x' - 1/2$. In either case, there exists an $x < \t$ such that $d(x) \geq \frac{A}{2}\log_\g n$ and at least $d(x)/2$ processes are active at time $x + 1/2$. If $x \geq \t-1/2$ then $d(\t) \geq d(x) / 2$, so suppose $x < \t - 1/2$.

Suppose $\mP_i$ is a process which is active at times $x$ and $x + 1/2$. The probability that $\mP_i$ has at least one arrival in $(x,x + 1/2)$ is $1 - e^{-\a/2}$. Suppose $\mP_i$ has an arrival at time $x_i\in (x, x+1/2)$. Then the process starting at time $x_i$ can be seen as the initial process of a CMJ process $\mC_i$ on $[x_i, \t]$. Since $\a > 1$, the probability that $\mC_i$ is active at time $\t$ is $q(\t-x_i) \geq 1 - \zeta$ (see Lemma \ref{constlem} (\ref{decreasing}) and (\ref{largealpha})). In other words, if $X_i$ is the indicator variable for $\mP_i$ having an active descendant at time $\t$, then $\Prob{X_i = 1} \geq (1-e^{-\a/2})(1-\zeta)$. This is true independently for the $d(x)/2$ processes $\mP_1,\dots,\mP_{d(x)/2}$ active at time $x$ and $x+1/2$, and we have $d(\t) \geq X_1 + \dots + X_{d(x)/2}$. Choosing $A, B$ large enough, Hoeffding's inequality \cite{hoef} shows that $d(\t) \geq d(x) / B$ with probability $1 - o(n^{-1})$. This finishes the proof. \end{proof}

\end{document}